\documentclass[12pt]{article}
\usepackage{amsfonts,amssymb,latexsym,amsmath}
\usepackage{eucal}
\usepackage{geometry}
\usepackage{hyperref}
\usepackage{bigints}
\usepackage{soul}
\usepackage{lineno}
\usepackage{xcolor}
\usepackage{lipsum}
\usepackage{upgreek}
\usepackage[draft]{todonotes}
\usepackage{amsthm}
\usepackage{mathtools}

\mathtoolsset{showonlyrefs,showmanualtags}

\newtheorem{theorem}{Theorem}
\newtheorem{proposition}[theorem]{Proposition}

\newtheorem{lemma}[theorem]{Lemma}

\def\<{\langle}
\def\>{\rangle}

\newcommand{\Ric}{\mathrm{Ric}}

\newcommand\be{\begin{equation}} 
	\newcommand\ee{\end{equation}}
\newcommand{\comment}[1]

\begin{document}

	\title{Modified mean curvature flow of graphs in Riemannian manifolds}
	\author{Jocel F. N. de Oliveira, Jorge H. S. de Lira\thanks{Partially supported by CNPq.}\, \& Matheus N. Soares\thanks{Partially supported by CNPq and FACEPE.}}
	\date{}
	\maketitle
	
	\begin{abstract} 
		We obtain height, gradient, and curvature \emph{a priori} estimates for a modified mean curvature flow in Riemannian manifolds endowed with a Killing vector field. As a consequence, we prove the existence of smooth, entire, longtime solutions for this extrinsic flow with smooth initial data.
	\end{abstract}
	
	\section{Introduction} 
	    
    The mean curvature flow (\textit{MCF}) is a classical subject in Geometric Analysis and is modeled in terms of a widely studied geometric evolution equation (\cite{Andrews:20}, \cite{Ecker:96}, \cite{Mantegazza:11}, \cite{Zhu:02})  which can be interpreted variationally as a sort of  $L^2$-gradient descent for the the area functional.

    Our goal in this paper is to prove the long-time existence of entire smooth graphs evolving by a \emph{modified} mean curvature flow in a Riemannian complete non-compact manifold $\bar M$ endowed with a Killing vector field.
    
    We suppose that the Killing vector $X$ in $\bar M$ never vanishes and that the orthogonal distribution determined by $X$ is integrable. Let $M$ be a given complete integral leaf of this distribution, and let $\Phi: M \times \mathbb{R}\to \bar M$  be the flow generated by $X$ with initial values in $M$.  It can be easily verified that $\Phi$ is a global isometry between $\bar M$ and the warped product $M\times_\varrho \mathbb{R}$ endowed with the static warped metric
	\begin{align}\label{warped}
		\varrho^2(x)  \,{\rm d}s^2 + g,
	\end{align}
	where $s$ is the natural coordinate in $\mathbb{R}$ and $g$ is the induced Riemannian metric in $M\subset \bar M$. The coefficient $\varrho \in C^\infty(M)$ is the norm of $X$ that is preserved along its flow lines. In these coordinates, we have $X = \partial_s$ and the leaves $\textcolor{black}{M_s:=\Phi_s (M)}$ are totally geodesic hypersurfaces in $\bar M$, \textcolor{black}{where $\Phi_s(\cdot)=\Phi(\cdot,s)$.}

    We suppose that the integral leaf $M$ is a manifold with a pole $o$. In this case we denote by $r$ the radial coordinate given by the geodesic distance in $M$ measured from $o$. Given that condition, we assume that the radial sectional curvatures of $M$ satisfy
\begin{align}
\label{cond-1}
    K (\partial_r \wedge {\sf v}) \ge  - \frac{\xi''(r)}{\xi(r)}
    \end{align}
    for all $r>0$ and ${\sf v}$ tangent to the geodesic sphere of $M$ with radius $r$ and centered at $o$. In \eqref{cond-1} the function $\xi\in C^\infty([0, \infty))$ is positive for $r>0$ and satisfies the following conditions
    \begin{equation}
        \label{cond-2}
        \xi'(0)=1, \,\, \xi^{(2k)}(0) = 0 \,\, \mbox{ for } k \in \mathbb{N}.
    \end{equation}
    This is a sufficient condition for the validity of a generalized version of the Hessian comparison theorem \cite{Pigola:05}. Finally, we suppose that 
\begin{equation}\label{cond-3}
		\left|\frac{\partial_r \varrho}{\varrho}\right| \le \frac{\xi'(r)}{\xi(r)}\cdot
	\end{equation}
Later on, we will interpret condition \ref{cond-3} in terms of the principal curvatures of cylinders in $\bar M$ ruled by the flow lines of $X$ over geodesic spheres of $M$. We refer to those cylinders as \emph{Killing cylinders}. 

We also suppose that there exists a constant $L>0$ such that the Ricci curvature tensor $\overline{\mbox{Ric}}$ satisfies
\begin{equation}
    \label{cond-4}
    \overline{\mbox{Ric}} \ge - L.
\end{equation}
The geometric setting considered in this paper includes, as particular examples, the Euclidean space $\mathbb{R}^{n+1}=\mathbb{R}^n \times \mathbb{R}$ and the hyperbolic space $\mathbb{H}^{n+1} = \mathbb{H}^n \times_{\varrho} \mathbb{R}$ where $\varrho(r) = \cosh r$ as well as the Riemannian product $\mathbb{H}^n \times \mathbb{R}$ where $\varrho=1$.

 Now, we define a suitable notion of graph in this geometric context. The \textit{Killing graph} of a smooth function $u_0: M\to \mathbb{R}$ is by definition the hypersurface $\Sigma_0$ in $\bar M$ given by
	\begin{equation}\label{eq_2.2}
		\Sigma_0 = \{\Phi(x,u_0(x)) :  x\in M\}.
	\end{equation}
	A smooth map $\widetilde\Psi: M \times [0, T) \to \bar M$, $T>0$, defines a \textcolor{black}{\textit{modified mean curvature flow}} with initial condition $\Sigma_0$ if each map $\widetilde \Psi_t := \widetilde \Psi(\cdot, t): M \to \bar M$, $t\in [0, T)$, defines an isometric immersion with $\widetilde \Psi_0 (M) = \Sigma_0$ and
	\begin{align}\label{MCF-mod}
		(\partial_t \widetilde \Psi)^ \perp= n (H-\sigma) N,
	\end{align}
	where $\sigma$ is a constant, $H$ is the \textcolor{black}{\textit{mean curvature}} \textcolor{black}{function} of $\Sigma_t = \widetilde \Psi_t(M)$ with respect to the unit normal vector field $N$ and $\perp$ indicates the orthogonal projection onto the normal bundle of $\Sigma_t$.

    The modified mean curvature flow models the gradient descent of a constrained area functional whose stationary immersions have constant mean curvature $\sigma$. We refer the reader to section \ref{non-par} for further details in the variational meaning of equation \ref{MCF-mod}.

     Having fixed the geometric conditions about the curvatures in $M$ and $\bar M$, we are able to state our main existence result for the modified mean curvature flow \eqref{MCF-mod}.

     \begin{theorem}\label{thm_0.1}
		Let $M$ be a $n$-dimensional complete, non-compact oriented  Riemannian manifold with a pole $o$ and let  $\bar{M}$ be the warped product $M\times_\varrho\mathbb{R}$ for some positive function $\varrho\in C^\infty(M)$. Suppose that the curvature conditions {\eqref{cond-1}}, {\eqref{cond-3}} and \eqref{cond-4} hold.  Given a locally Lipschitz entire Killing graph $\Sigma_0$ over $M$, there exists a smooth modified mean curvature flow {\eqref{MCF-mod}} with initial condition $\Sigma_0$ defined for all $t> 0$ and for all $\sigma$ satisfying
	\begin{equation}\label{eq_2.13}
		\sigma < \frac{1}{n}\inf \left( \frac{|\bar\nabla  \varrho|}{\varrho}+(n-1) \frac{\xi'(r)}{\xi(r)}\right).
	\end{equation}
	\end{theorem}

    Our proof of theorem \ref{thm_0.1} relies on the foundational contributions for the geometric analysis of the mean curvature flow ($\sigma =1$) in the Euclidean space  given by K. Ecker and G. Huisken in  \cite{Ecker:89} and \cite{Ecker:91}. There, the authors prove  longtime   existence and asymptotic behavior results for the MCF of graphs with locally Lipschitz initial data. 
    
    Theorem \ref{thm_0.1} also extends previous results by P. Untenberger in  \cite{Unterberger:98} and \cite{Unterberger:03} for the particular case of the mean curvature flow in the hyperbolic space $\mathbb{H}^{n+1}$ modeled as the warped product manifold $\mathbb{H}^n \times_{\varrho} \mathbb{R}$ with $\varrho=\cosh r$. 

    In \cite{Allmann:17}, Patrick Allmann, Longzin Lin and Jingyong Zhu have introduced the notion of modified mean curvature flow for starshaped hypersurfaces (in fact, Killing graphs) in $\mathbb{H}^{n+1}$ with a given asymptotic boundary. The main results in \cite{Allmann:17} concern longtime existence and asymptotic behavior of the modified mean curvature flow in the particular case of the hyperbolic space.    
    
In \cite{Borisenko:12}, A. Borisenko and V. Miquel have obtained longtime existence results for the mean curvature flow with locally Lipschitz initial conditions in the general setting of warped product manifolds. 

   The main contribution in this paper is to extend the results in \cite{Allmann:17} and \cite{Borisenko:12} into an unified setting of modified mean curvature flow of graphs in a warped product manifold of the form $M\times_\varrho \mathbb{R}$.
 
This paper is organized as follows: In section \ref{non-par} we describe the modified mean curvature flow of Killing graphs in terms of a quasilinear parabolic equation that characterizes critical immersions for a constrained area functional. In section \ref{Hessian} we deduce some fundamental differential inequalities from a suitable version of the Hessian comparison principle whose validity in $M$ follows from condition \ref{cond-1}. These inequalities are crucial for the proof of the \emph{a priori} estimates. Height bounds are established in section \ref{height-est} using a sort of prescribed mean curvature spherical caps as barriers. Interior gradient estimates are obtained along the same lines as in \cite{Ecker:89} and \cite{Ecker:91} in section \ref{C1-est}. Boundary gradient estimates are also presented in this section. Curvature and higher derivatives bounds are obtained in section \ref{C2-est}. Finally,  the proof of the existence theorem \ref{thm_0.1} is discussed in section \ref{proof-thm} using the  \emph{a priori} estimates obtained in the preceding sections.

\section{Variational and non-parametric formulations}\label{non-par}
	
    In this section, we verify that the modified mean curvature equation \ref{MCF-mod} may be written in terms of a quasilinear parabolic partial differential equation in $M$. This allows us to use the analytical framework established in references as \cite{Ladyzhenskaya:68},  \cite{Lieberman:96} and \cite{Huisken:99} to prove longtime existence results for smooth compact domains in $M$ once we have obtained \emph{a priori} estimates for height, gradient and curvature of the evolving graphs. 
	
	In order to obtain a non-parametric description of \ref{MCF-mod}, one might assume the existence of a one-parameter family of diffeomorphisms $\psi: M \times [0, T') \to M$  and a one parameter family of functions $u: M \times [0,T')\to \mathbb{R}$ for some $T' \le T$ so that
	\begin{equation}\label{rep-flow}
	\widetilde \Psi(\psi(x,t), t) = \Phi (x, u(x,t)) =: \Psi(x, t),
	\end{equation}
	for $(x, t)\in M\times [0, T')$. A suitable choice of $\psi$ yields
	\begin{equation}\label{MCF-non-par}
		\partial_t \Psi = n(H-\sigma) N,
	\end{equation}
	a modified mean curvature flow for which the evolving  hypersurfaces $\Sigma_t$ are Killing graphs 
	\begin{equation}\label{non-par-mcf}
		\Sigma_t = \Psi_t (M) = \{\Phi(x, u(x, t)): x\in M\}, \quad t\in [0, T').
	\end{equation}
	In this description, we have
	\begin{equation}\label{eq_2.7}
		N = N|_{\Psi(\cdot, \, t)} = \frac{1}{W} (\varrho^{-2} X - \Phi_* \nabla^M u)
	\end{equation} 
	with
	\begin{equation}\label{eq_2.8}
		W = (\varrho^{-2}+|\nabla^M u|^2)^{\frac{1}{2}},
	\end{equation}
	where $\nabla^M$ denotes the Riemannian connection in $(M, g)$. 

    Considering the non-parametric formulation of the modified mean curvature flow in \eqref{non-par-mcf}, one easily verifies that the induced metric and volume element in $\Sigma_t$, $t\in [0, T')$,  are respectively  given by
	\begin{equation}\label{eq_3.1}
		g+\varrho^2 {\rm d}u \otimes {\rm d}u
	\end{equation}
	and
	\begin{equation}\label{eq_3.2}
		{\rm d} \Sigma_t=\varrho\sqrt{\varrho^{-2}+|\nabla^M u|^2}\, {\rm d}M.
	\end{equation}
	Given a domain $\Omega \subset M$ and a constant $\sigma$ 
	we define the constrained area functional 
	\begin{equation}\label{eq_3.3}
		\mathcal{A}_\sigma\left[u\right]=\int_{\Omega}\varrho\sqrt{\varrho^{-2}+|\nabla^M u|^2}\, {\rm d}M+n\sigma\mathcal{V}\left[u\right],
	\end{equation}
	where the volume functional $\mathcal{V}$ is defined by
	\begin{equation}\label{eq_3.4}
		\mathcal{V}\left[u\right]=\int_\Omega\int_0^{u(x)\varrho(x)}  {\rm d}M=\int_\Omega \varrho u\, {\rm d}M.
	\end{equation}
	For an arbitrary compactly supported function $v\in C_0^\infty(\Omega)$ we have
	\[
	\frac{{\rm d}}{{\rm d} t}\Big|_{t=0}\mathcal{A}_\sigma\left[u+t v\right]\text{\small $=-\int_\Omega\left(\operatorname{div}_M\left(\frac{\nabla^M u}{W}\right)+\left\langle\nabla^M\log\varrho,\frac{\nabla^M u}{W}\right\rangle -n\sigma\right)v\varrho\, {\rm d}M,$}
	\]
	where the differential operators $\nabla^M$ and $\operatorname{div}_M$ are taken with respect to the metric $g$ in $M$. Then the Euler-Lagrange equation of the functional $\mathcal{A}_\sigma$ is
	\begin{equation}\label{H-1}
		n(H-\sigma) = \operatorname{div}_M\left(\frac{\nabla^M u}{W}\right)+\left\langle\nabla^M\log\varrho,\frac{\nabla^M u}{W}\right\rangle -n\sigma =0,
	\end{equation}
	where $H$ is the mean curvature of the Killing graph of $u$. However, differentiating the right-hand side of \eqref{rep-flow} with respect to $t$ in  yields
	\[
	\partial_t \Psi = \partial_t u \, X.
	\]
	Hence, \eqref{H-1} is equivalent to 
	\begin{equation}\label{eq_3.6}
		\partial_t u \, X = \left( \operatorname{div}_M\left(\frac{\nabla^M u}{W}\right)+\left\langle\nabla^M\log\varrho,\frac{\nabla^M u}{W}\right\rangle -n\sigma\right) N.
	\end{equation}
	Taking the coefficient of normal projection on both sides yields
	\begin{equation}\label{eq_3.7}
	\partial_t u \langle X, N\rangle  = \operatorname{div}_M\left(\frac{\nabla^M u}{W}\right)+\left\langle\nabla^M\log\varrho,\frac{\nabla^M u}{W}\right\rangle -n\sigma.
	\end{equation}
	Since $\langle X, N\rangle = 1/W$ we conclude that the map $\Psi$ in \eqref{rep-flow} defines a modified mean curvature flow if and only
	if $u(\cdot, \, t)$ satisfies the parabolic equation
	\begin{equation}\label{eq_3.8}
		\partial_t u =  \mathcal{Q}[u],
	\end{equation}
	where
	\begin{equation}\label{eq_3.9}
		\mathcal{Q}[u]=W\left(\operatorname{div}_M\left(\frac{\nabla^M u}{W}\right)+\left\langle\nabla^M\log\varrho,\frac{\nabla^M u}{W}\right\rangle-n\sigma\right).
	\end{equation}
	In general, this non-parametric formulation is equivalent to the modified mean curvature flow \eqref{MCF-mod} up to tangential diffeomorphisms of the evolving graphs $\Sigma_t$, $t\in [0, T')$. The equivalence here follows from the fact that we are assuming a fixed \emph{gauge}, namely the choice of coordinates fixed in \eqref{rep-flow}.

	\section{Hessian comparison principle and fundamental inequalities}\label{Hessian}
	
	Our assumption \eqref{cond-1} about radial sectional curvatures in $M$ along  radial geodesics issuing from the pole $o$ implies the validity of  the Hessian comparison theorem  in the following version
	\begin{equation}\label{hess-comp}
		\nabla^M \nabla^M r \le \frac{\xi'(r)}{\xi(r)} \left(g-{\rm d}r\otimes {\rm d}r\right).
	\end{equation}
    We refer the reader to \cite{Pigola:05} for a proof of this fact.

	 Now we deduce from \eqref{hess-comp} some differential inequalities that will be useful to prove \emph{a priori} estimates in the subsequent sections. One of the relevant functions in those inequalities is
     \begin{equation}
		\label{bar-zeta}
		\bar\zeta(r) = \int^{r}_0 \xi(\varsigma)\, {\rm d}\varsigma.
	 \end{equation}
     In the sequel, this function plays the role of the quadratic function $x\mapsto r^2 = | x|^2$ in \cite{Ecker:89} and \cite{Ecker:91}.
     
	\begin{proposition}\label{hess-r} Suppose that \eqref{MCF-mod} holds. The restrictions of the functions $r$ and $s$ to the hypersurfaces $\Sigma_t= \widetilde\Psi_t (M)$, $t\in [0, T)$, satisfy
		\begin{equation}\label{eq_3.10}
			\begin{split}
				(\partial_t -\Delta) r &\ge -\frac{\xi'(r)}{\xi(r)} \left(n- |\nabla r|^2\right)- \varrho^2|\nabla s|^2 \left(\langle\bar\nabla\log\varrho, \nabla r\rangle -\frac{\xi'(r)}{\xi(r)}\right)
                \\
                &-n\sigma\langle\bar\nabla r,N\rangle
			\end{split}
		\end{equation}
		and
		\begin{equation}
			\label{eq_3.11}
			(\partial_t - \Delta )s = -(2\langle\bar{\nabla}\log\varrho,N\rangle+n\sigma)\langle\bar{\nabla}s,N\rangle.
		\end{equation}
		In both expressions, $\nabla$ and $\Delta$ are the intrinsic Riemannian connection and Laplacian in $\Sigma_t$, respectively, whereas $\bar\nabla$ denotes the Riemannian connection in $\bar M$. Moreover, given the function
		\begin{equation}
			\label{eq_3.12}
			\zeta(\widetilde\Psi(x, t)) = \int_0^{r(\widetilde\Psi(t,x))} \xi(\varsigma)\, {\rm d}\varsigma,
		\end{equation}
		that is, $\zeta = \bar\zeta (r\circ\widetilde\Psi)$, we have
		\begin{equation}
			\label{eq_3.13}
			\begin{split}
				(\partial_t -\Delta) \zeta &\ge -n\xi'(r) - \varrho^2|\nabla s|^2 \xi(r) \left(\langle\bar\nabla\log\varrho, \nabla r\rangle -\frac{\xi'(r)}{\xi(r)}\right)
                \\
                &-n\sigma\xi(r)\langle\bar\nabla r,N\rangle.\end{split}
		\end{equation}
	\end{proposition}
	
	\begin{proof} 
		Observe that  $\bar \nabla s = \varrho^{-2}X$ and $\nabla s = \varrho^{-2}X^\top$, where $\top$ denotes the tangential projection onto $T\Sigma_t$. Given a local orthonormal tangent frame $\{{\sf e}_i\}_{i=1}^n$ in $\Sigma_t$, one has
		\begin{equation}\label{eq_3.14}
			\begin{split}
				\Delta s = 2\langle\bar{\nabla}\log\varrho,N\rangle\langle \bar\nabla s, N\rangle+nH\langle\bar{\nabla}s,N\rangle.
			\end{split}
		\end{equation}
		We also compute
		\begin{equation}\label{eq_3.15}
			\partial_t s=\langle\bar{\nabla}s, \partial_t \widetilde\Psi\rangle=n(H-\sigma)\langle\bar{\nabla}s,N\rangle.
		\end{equation}
		The Hessian of the radial coordinate satisfies
		\begin{equation}\label{eq_3.17}
			\langle \bar\nabla_X \bar\nabla r, X\rangle = \langle \bar\nabla_{\bar\nabla r} X, X\rangle = \frac{1}{2}\partial_r |X|^2 = \frac{1}{2}\partial_r \varrho^2  =
		\varrho \langle\bar\nabla\varrho, \bar\nabla r\rangle.
		\end{equation}
		Fixed a local orthonormal tangent frame $\{{\sf e}_i\}_{i=1}^n$ in $\Sigma_t$, we have
		\begin{equation}\label{eq_3.18}
                \Delta r = \sum_i\langle\nabla_{{\sf e}_i}\nabla r,{\sf e}_i\rangle =  \sum_i \langle\nabla^M_{\pi_*{\sf e}_i}\nabla^M r,\pi_* {\sf e}_i\rangle+ |\nabla s|^2 \langle\varrho\bar\nabla\varrho, \nabla r\rangle+nH\langle\bar\nabla r,N\rangle
		\end{equation}
		where $\pi: \bar M = M \times \mathbb{R} \to M$ is the projection on the first factor, that is, $\pi(s,x) =x$ for all $(s,x)\in M\times \mathbb{R}$. The Hessian comparison theorem \eqref{hess-comp} implies that
		\begin{equation}\label{eq_3.19}
			\begin{split}
				\Delta r  \le \frac{\xi'(r)}{\xi(r)} \left(n- \varrho^2|\nabla s|^2 - |\nabla r|^2\right)+ \varrho^2|\nabla s|^2 \langle\bar\nabla\log\varrho, \nabla r\rangle+nH\langle\bar\nabla r,N\rangle.
			\end{split}
		\end{equation}
		Hence,
		\begin{equation}\label{eq_3.20}
			\Delta r \le  \frac{\xi'(r)}{\xi(r)} \left(n-  |\nabla r|^2\right)+ \varrho^2|\nabla s|^2\left( \langle\bar\nabla\log\varrho, \nabla r\rangle- \frac{\xi'(r)}{\xi(r)}\right)+nH\langle\bar\nabla r,N\rangle.
		\end{equation}
		We also have $\nabla \zeta = \xi(r)\nabla r$ and 
		\begin{equation}\label{eq_3.21}
			\begin{split}
				\Delta \zeta&= \xi(r) \Delta r +\xi'(r) |\nabla r|^2. 
			\end{split}
		\end{equation}
		Therefore
		\begin{equation}\label{eq_3.22}
			\begin{split}
				&  \Delta \zeta \le n\xi'(r) + \varrho^2|\nabla s|^2 \xi(r)\left( \langle\bar\nabla\log\varrho, \nabla r\rangle- \frac{\xi'(r)}{\xi(r)}\right)+nH\xi(r)\langle\bar\nabla r,N\rangle.
			\end{split}
		\end{equation}
		On the other hand,
		\begin{equation}\label{eq_3.23}
			\partial_t r = \big\langle \bar\nabla r, \partial_t\widetilde\Psi\big\rangle = n(H-\sigma) \langle \bar\nabla r, N\rangle
		\end{equation}
		and
		\begin{equation}\label{eq_3.24}
			\partial_t \zeta = n(H-\sigma) \xi(r)\langle \bar\nabla r, N\rangle.
		\end{equation}
		We conclude that 
		\begin{equation}\label{eq_3.25}
			\begin{split}
				(\partial_t -\Delta) r \ge -\frac{\xi'(r)}{\xi(r)} \left(n- |\nabla r|^2\right)- \varrho^2|\nabla s|^2 \left(\langle\bar\nabla\log\varrho, \nabla r\rangle -\frac{\xi'(r)}{\xi(r)}\right)-n\sigma\langle\bar\nabla r,N\rangle
			\end{split}
		\end{equation}
		and
		\begin{equation}\label{eq_3.26}
			\begin{split}
				(\partial_t -\Delta) \zeta \ge -n\xi'(r) - \varrho^2|\nabla s|^2 \xi(r) \left(\langle\bar\nabla\log\varrho, \nabla r\rangle -\frac{\xi'(r)}{\xi(r)}\right)-n\sigma\xi(r)\langle\bar\nabla r,N\rangle.
			\end{split}
		\end{equation}
		This finishes the proof of the proposition.
	\end{proof}

The next two applications of the Hessian comparison theorem provide $C^1$ bounds for \eqref{MCF-mod}. 

These propositions are of independent interest since they provide suitable adaptations of some gradient estimates in \cite{Korevaar:86}, \cite{Ecker:89} and \cite{Ecker:91} to the more general context of modified mean curvature flow in Riemannian warped products.
    \begin{proposition}\label{evolutionW} If the hypersurfaces $\Sigma_t$, $t\in [0, T)$, evolve by the modified mean curvature flow \eqref{MCF-mod}, then
		\begin{equation}\label{eq_3.27}
			\left(\partial_t-\Delta\right)W=-W(|A|^2+\overline{\Ric}(N,N))-2W^{-1}|\nabla W|^2,
		\end{equation}
		where $W = \langle X, N\rangle^{-1}$ {\rm(}note that $W=(\varrho^{-2}+|\nabla^M u)|^2)^{1/2}$ whenever $\Sigma_t$ is a Killing graph{\rm)} and $A$ is the Weingarten map of $\Sigma_t$. If the ambient Ricci tensor satisfies $\overline{\rm Ric} \ge - L$ for some non-negative constant $L$ then
		\[
		(\partial_t-\Delta) (e^{-Lt} W) \le 0.
		\]
		In particular, the parabolic maximum principle implies in this case that
		\begin{equation}\label{eq_3.28}
			{\sup}_{B_R(o) \times [0, T]}W (x,t) \le  e^{LT} \big({\sup}_{B_R(o)} W(\cdot,  0) + {\sup}_{\partial B_R(o) \times (0, T]} W\big).
		\end{equation}
	\end{proposition}
	\begin{proof} Note that
		\begin{equation}\label{eq_3.29}
			\nabla\langle X,N\rangle=\langle X,N\rangle (\bar\nabla\log\varrho)^\top-\langle\bar\nabla\log\varrho,N\rangle X^\top-AX^\top.
		\end{equation}
		Hence,
		\begin{equation}\label{eq_3.30}
			\begin{split}
				& \langle X,N\rangle (\bar\nabla\log\varrho)^\top-\langle\bar\nabla\log\varrho,N\rangle X^\top = \langle X, N\rangle \bar\nabla\log\varrho - \langle\bar\nabla\log\varrho, N\rangle X.
			\end{split}
		\end{equation}
		It follows from the second variation formula for the functional $\mathcal{A}_0$ that 
		\begin{equation}\label{eq_3.31}
			\Delta\langle X,N\rangle+|A|^2\langle X,N\rangle+\overline{\Ric}(N,N)\langle X,N\rangle=-n\langle\nabla H,X^\top\rangle,
		\end{equation}
		where $\top$ denotes the tangent projection onto $T\Sigma_t$. 
		On the other hand, using that $X$ is a Killing vector field one gets
		\begin{equation}\label{eq_3.32}
			\begin{split}
				\partial_t\langle X,N\rangle &= n(H-\sigma)\langle \bar\nabla_N X, N\rangle - n \langle X, \nabla (H-\sigma)\rangle=-n\langle X^\top,\nabla H\rangle,
			\end{split}
		\end{equation}
		where $\bar{\nabla}$ denotes the Riemannian connection in $\bar{M}$ and so
		\begin{equation}\label{eq_3.33}
			\left(\partial_t-\Delta\right)\langle X,N\rangle=|A|^2\langle X,N\rangle+\overline{\Ric}(N,N)\langle X,N\rangle.
		\end{equation}
		Thus,  using that  $\langle X,N\rangle=1/W$ one has
		\begin{equation}\label{eq_3.34}
			\partial_tW = -W^2 \partial_t W^{-1}=-W^2\partial_t\langle X,N\rangle
		\end{equation}
		and
		\begin{equation}\label{eq_3.35}
			\Delta W-\frac{2}{W}|\nabla W|^2 =-W^2 \Delta W^{-1} = -W^2 \Delta\langle X, N\rangle.
		\end{equation}
		Hence, one concludes that
		\begin{equation}\label{eq_3.36}
			\partial_t W-\Delta W+\frac{2}{W}|\nabla W|^2
			=-W^2\left(\partial_t-\Delta\right)\<X,N\>
			=-W(|A|^2+\overline{\Ric}\left(N,N)\right)
		\end{equation}
		A direct application of the parabolic maximum principle \cite{Mantegazza:11} finishes the proof of the proposition.
	\end{proof}
	
	\vspace{3mm}
	
	The following proposition is analog to Theorem 2.1 in \cite{Ecker:89} and provides a localized version of the gradient bound in \eqref{eq_3.28}.
	
	\begin{proposition} Suppose that the ambient Ricci tensor satisfies $\overline{\rm Ric} \ge - L$ for some non-negative constant $L$. Given $R>0$, denote
		\begin{equation}\label{eq_3.37}
			C_R = {\rm sup}_{B_R(o)} \left( n \xi'(r)+\partial_r \log \varrho + n\sigma \xi(r)\right).
		\end{equation}
		If the hypersurfaces $\Sigma_t$, $t\in [0, T)$, evolve by the modified mean curvature flow \eqref{MCF-mod}, then the function $W = \langle X, N\rangle^{-1}$  satisfies
		\begin{equation}\label{eq_3.38}
			e^{-Lt} (\bar\zeta(R) - \zeta(\widetilde \Psi(x,t)) - C_Rt)W(x,t) \le  \bar\zeta(R) \, {\sup}_{M} W(\cdot, 0) 
		\end{equation}
		whenever 
		\begin{equation}\label{eq_3.39}
			\zeta (\widetilde\Psi(x,t)) \le \bar\zeta(R) - C_R T.
		\end{equation}
	\end{proposition}
	\begin{proof} Given $R>0$, let 
		\begin{equation}\label{eq_3.40}
			\phi (\varsigma) = (\bar\zeta(R)- \varsigma)^2, \,\, \mbox{ for } \,\, \varsigma \le \bar\zeta(R)
		\end{equation}
		and
		\begin{equation}\label{eq_3.41}
			\bar\eta(\bar\zeta,t) = e^{-2Lt} \phi (\bar\zeta + C_R t).
		\end{equation}
		Denoting
		\begin{equation}\label{eq_3.42}
			\eta = \bar\eta (\zeta(\widetilde \Psi(x,t)), t),
		\end{equation}
		and using \eqref{eq_3.13} it follows from the fact that $\phi'<0$ that
		\begin{equation}\label{eq_3.43}
			\begin{split}
				 (\partial_t - \Delta) \eta \le -2L \eta  - e^{2Lt}\frac{\phi''}{\phi'^2} |\nabla \eta|^2.
			\end{split}
		\end{equation}
		Hence, \eqref{eq_3.27} implies that
		\begin{equation}\label{eq_3.44}
			\begin{split}
				& (\partial_t - \Delta) (\eta W^2) \le -6\eta |\nabla W|^2 - 2\langle \nabla\eta, \nabla W^2\rangle-\frac{\phi\phi''}{\phi'^2} \frac{|\nabla \eta|^2}{\eta} W^2.
			\end{split}
		\end{equation}
		Following \cite{Ecker:89}, we write
		\begin{equation}\label{eq_3.45}
			\begin{split}
				& - 2\langle \nabla\eta, \nabla W^2\rangle = -6W\langle \nabla \eta, \nabla W\rangle +  \bigg\langle \frac{\nabla\eta}{\eta}, \nabla (\eta W^2)\bigg\rangle \frac{|\nabla \eta|^2}{\eta}W^2.
			\end{split}
		\end{equation}
		Replacing this above one concludes  that
		\small{\begin{equation}\label{eq_3.46}
			\begin{split}
				(\partial_t - \Delta) (\eta W^2) - \bigg\langle \frac{\nabla\eta}{\eta}, \nabla (\eta W^2)\bigg\rangle  
				&\le -6\eta |\nabla W|^2 -6W\langle \nabla \eta, \nabla W\rangle 
                \\
                &-\left(1+\frac{\phi\phi''}{\phi'^2}\right)\frac{|\nabla \eta|^2}{\eta}W^2.
			\end{split}
		\end{equation}}
		Using Young's inequality as in \cite{Ecker:89} we write
		\[
		-6W\langle \nabla \eta, \nabla W\rangle \le 6 \eta|\nabla W|^2 + \frac{3}{2} \frac{|\nabla \eta|^2}{\eta} W^2
		\]
		It follows from the choice of $\phi$ that
		\begin{equation}\label{eq_3.47}
			(\partial_t - \Delta) (\eta W^2) - \bigg\langle \frac{\nabla\eta}{\eta}, \nabla (\eta W^2)\bigg\rangle  
			\le 0
		\end{equation}
		and the conclusion is a direct consequence of the parabolic maximum principle.
	\end{proof}

\section{Height estimates}\label{height-est}

The \emph{a priori} $C^0$ bounds for the modified mean curvature flow \ref{MCF-mod} depend on geometric barriers mimicking the role of spherical caps in the Euclidean space. In what follows, we construct an example of such barriers.

\subsection{Examples of modified MCF in rotationally symmetric metrics}\label{examples-MCF}
    
	Let $M_+$ be a complete, non-compact, $n$-dimensional model manifold with respect to a fixed pole $o_+\in M_+$ in the sense that the Riemannian metric in $M_+$ can be expressed in Gaussian coordinates $(r, \vartheta)\in \mathbb{R}\times \mathbb{S}^{n-1}$ centered at $o_+$ as
	\begin{equation}\label{eq_3.49}
		g_+ = {\rm d}r^2 + \xi^2_+(r)\, {\rm d}\vartheta^2
	\end{equation} 
	where ${\rm d}\vartheta^2$ denotes the round metric in $\mathbb{S}^{n-1}$ and $\xi_+\in C^\infty ([0, \infty))$ satisfies condition \eqref{cond-2}. We also consider a smooth radial function $\varrho_+\in C^\infty([0, \infty))$ satisfying
	\begin{equation}\label{eq_3.50}
		\begin{split}
			& \varrho_+(r)> 0, \,\, \mbox{ for }\,\, r>0,  \\
			& \varrho_+(0) = 1,\\
			& \varrho_+^{(2k+1)} =0, \,\, \mbox{ for }\,\, k \in \mathbb{N}
		\end{split}
	\end{equation}
	and define the warped metric in $M_+ \times \mathbb{R}$ 
	\begin{equation}\label{eq_3.51}
		\varrho^2_+(r) {\rm d}s^2 + {\rm d}r^2 + \xi^2_+(r)\, {\rm d}\vartheta^2.
	\end{equation}
	We denote
	\begin{equation}\label{eq_3.52}
		A(r) =   \varrho_+(\varsigma)\xi^{n-1}_+(\varsigma)
	\end{equation}
	and
	\begin{equation}\label{eq_3.53}
		V(r) = \int_0^r \varrho_+(\varsigma)\xi^{n-1}_+(\varsigma)\, {\rm d}\varsigma.
	\end{equation}
	We also define
	\begin{equation}\label{eq_3.54}
		H(r) = -\frac{1}{n}\frac{A(r)}{V(r)}\cdot
	\end{equation}
	Given $x\in M_+$ we denote $r(x) = {\rm dist}(o_+,x)$. Let $B_R(o_+)$ be the closed geodesic ball centered at $o_+$ with radius $R\in (0,\infty)$. Hence $x\in B_R(o_+)$ if and only if $r(x) \le R$. The mean curvature of the Killing cylinder over the geodesic sphere $\partial B_r(o_+)$ is given by
	\begin{equation}\label{eq_3.55}
		H_{\rm cyl}(r) = \frac{1}{n} \left((n-1)\frac{\xi_+'(r)}{\xi_+(r)}+\frac{\varrho_+'(r)}{\varrho_+(r)}\right).
	\end{equation}
	Fixed this notation, we are able to state the following result.
	
	\begin{proposition}
    \label{caps}
		For each $R\in (0,\infty)$, the graph of the function 
		\begin{equation}\label{eq_3.56}
			v_R(x)  =\int^{r(x)}_R \frac{nH(R)V(\varsigma)}{\varrho(\varsigma) (A^2(\varsigma)-n^2 H^2(R)V^2(\varsigma))^{\frac{1}{2}}}\,{\rm d}\varsigma
		\end{equation}
		defined in $B_R(o_+)$ has constant mean curvature $H(R)$ and its boundary is the geodesic sphere $\partial B_R(o_+)$ in the model manifold $M_+\times_{\varrho_+} \mathbb{R}$.
	\end{proposition}

	\begin{proof} Fixed $R\in (0,\infty)$ let $v_R$ be the (radial) solution of the following Dirichlet problem for the constant mean curvature equation $M_+\times_{\varrho_+} \mathbb{R}$, namely
	\begin{equation}\label{eq_3.57}
		\begin{cases}
			& \operatorname{div}_+\left(\frac{\nabla^+ v_R}{W_+}\right)+g_+\left(\nabla\log\varrho_+,\frac{\nabla^+ v_R}{W_+}\right)=nH(R) \,\,\mbox{ in } \,\, B_R(o_+),\\
			& v_R|_{\partial B_R(o_+)} = 0,
		\end{cases}
	\end{equation}
	where the differential operators  ${\rm div}_+$ and $\nabla^+$ are defined with respect to the metric \eqref{eq_3.49} in $M_+$ and
	\[
	W_+ = (\varrho^{-2}_+(r)+ v'^2_R(r))^{\frac{1}{2}},
	\]
	with $'$ denoting derivatives  with respect to $r$.  Note that \eqref{eq_3.57} can be written in terms of a weighted divergence as
	\begin{equation}\label{eq_3.58}
		\operatorname{div}_{-\log\varrho_+} \left(\frac{\nabla^+ v_R}{W_+}\right) = \frac{1}{\varrho_+} \operatorname{div}_+ \left(\varrho_+\frac{\nabla^+ v_R}{W_+}\right) = nH(R).
	\end{equation}
	Hence, divergence theorem yields
	\begin{equation}\label{eq_3.59}
		\begin{split}
			\int_{B_r(o)} \textcolor{black}{\varrho_+} nH(R) \, {\rm d}M_+ &= \int_{B_r(o)} \operatorname{div}_+ \left(\varrho_+\frac{\nabla^+ v_R}{W}\right)\, {\rm d}M_+ \\
			& = \int_{\partial B_r(o)} g_+\left( \frac{\nabla^+ v_R}{W_+}, \partial_r\right) \varrho_+\, {\rm d}\partial B(r),
		\end{split}
	\end{equation}
	for $r\le R$. Since  $v_R$ is radial \eqref{eq_3.57} becomes
	\begin{equation}\label{eq_3.60}
		\begin{split}
			& \left(\frac{v'_R(r)}{(\varrho^{-2}_+(r)+v'^2_R(r))^{1/2}}\right)'+\frac{v'_R(r)}{(\varrho^{-2}_+(r)+v'^2_R(r))^{1/2}}\left(\frac{\varrho'_+(r)}{\varrho_+(r)}+(n-1)\frac{\xi'_+(r)}{\xi_+(r)}\right) \\
			& = nH(R).
		\end{split}
	\end{equation}
	It follows from \eqref{eq_3.59}  that  $v_R$ is the solution of the first order equation
	\begin{equation}\label{eq_3.61}
		\frac{v'_R(r)}{(\varrho^{-2}_+(r)+v'^2_R(r))^{1/2}} \,\varrho_+(r) \xi^{n-1}_+(r) =
		\int_0^r nH(R) \varrho_+(\varsigma) \xi^{n-1}_+(\varsigma)\, {\rm d}\varsigma.
	\end{equation}
	with initial condition $v_R|_{r=R}=0$. 
	Solving this expression for $v'_R$, one obtains
	\begin{equation}\label{eq_3.62}
		v'_R(r)=\frac{nH(R)V(r)}{\varrho_+(r)(A^2(r)-n^2 H^2(R)V^2(r))^{1/2}}\cdot
	\end{equation}
	The graph $\Sigma_R$ of $v_R$ is a rotationally invariant hypersurface which can be parametrized using coordinates $(s,r,\vartheta)$
	as $\varsigma\mapsto  (s(\varsigma), \vartheta, r(\varsigma))$, where $\varsigma$ can be taken as the arc lenght parameter. Denoting by $\phi$ the angle between the coordinate vector field $\partial_r$ and a given profile curve $\vartheta={\rm constant}$ in $\Sigma_R$ one has
	\begin{equation}\label{eq_3.63}
		\dot r = \cos\phi, \quad \varrho \dot s = \sin\phi.
	\end{equation}
	Thus,  \eqref{eq_3.60} becomes
	\begin{equation}\label{eq_3.65}
		\frac{{\rm d}\phi}{{\rm d}\varsigma} +\sin\phi\left(\frac{\varrho'(r)}{\varrho(r)}+(n-1)\frac{\xi'(r)}{\xi(r)}\right) = -nH(R).
	\end{equation}
	Hence, a profile curve of $\Sigma_R$ is given by the solution of the first order system
	\begin{equation}\label{eq_3.66}
		\begin{split}
			& \dot r = \cos\phi,\\
			& \varrho \dot s = \sin\phi,\\
			& \dot \phi =- nH(R) - nH_{{\rm cyl}}(r)\sin\phi,
		\end{split}
	\end{equation}
	with initial conditions $r(0) = R, s(0) = 0, \phi(0) = \frac{\pi}{2}$. In this case \eqref{eq_3.61} can be rewritten as
	\begin{equation}\label{eq_3.67}
		\varrho(r) A(r)\dot s =-nH(R) V(r)
	\end{equation}
	where $\cdot$ indicates derivatives with respect to the parameter $\varsigma$. Hence, it is immediate that when the coordinate $r$ attains its maximum, that is, when $r=R$, we have $\dot r=0$ and $\dot s=1$. This is consistent with the choice of $H(R)$ in 
	\eqref{eq_3.54}. We also observe that $\dot s \to 0$ and $\dot r \to 1$ as $r \to 0^+$.
	\end{proof}
	
	\vspace{3mm}
	For $R\ge r_0$, note that  the variable $\mu = R-r_0$ can be considered as the geodesic distance between the geodesic spheres $\partial B_{r_0}(o) = \partial\Sigma_{r_0}$ and $\partial B_R(o) =\partial\Sigma_R$. Hence, $\nabla \mu|_{\partial B_R(o)} =\partial_r|_{r=R}$.  Thus, we set a time parameter $t \in [0,\infty)$ given by
	\begin{equation}\label{eq_3.68}
		\begin{split} 
			& \frac{{\rm d}\mu}{{\rm d}t} = -n(H(R)-\sigma) = -n(H(\mu+r_0)-\sigma),\\
			& \mu(0) = 0.
		\end{split}
	\end{equation}
	Hence, $\mu=\mu(t)$ is implicitly defined by
	\[
	\int_{r_0}^{\mu(t)+r_0} \frac{V(\varsigma)}{A(\varsigma)}\,{\rm d}\varsigma = t
	\]
	Denote $R(t) = \mu(t) + r_0$. We claim that the one-parameter family of constant mean curvature graphs 
	$\{\Sigma_{R(t)}\}_{t\ge 0}$ evolves by the  (negative) mean curvature flow 
	\begin{equation}
		\label{eq_3.69}
		\partial_t \Psi^+ = -n(H(R(t))-\sigma)N_t,
	\end{equation}
	where
	\begin{equation}\label{eq_3.70}
		N_t= \frac{1}{W} (\varrho^{-2}(r) X - v'_R(r)\partial_r) = -\frac{\dot r}{\varrho} X +\varrho \dot s\, \partial_r.
	\end{equation}
	This means that $\Sigma_{R(t)} = \Psi_t^+ (\Sigma_{r_0})$. 
	In particular, we must have
	\begin{equation}\label{eq_3.71}
		\partial B_{R(t)} =\partial \Sigma_{R(t)} = \Psi^+_t (\partial\Sigma_{r_0}) = \Psi^+_t (\partial B_{r_0}).
	\end{equation}
	In other terms, the time parameter $t$ must be chosen in such a way that the geodesic spheres evolve as $\partial B_{R(t)}=\Psi^+_t(\partial B_{r_0})$, where $\Psi^+$ is the isometric immersion of $M_+$ as  in the definition of $\eqref{MCF-non-par}$. Since $\dot r =0$ and $\varrho \dot s=1$ at $r=R(t)$ it follows from \eqref{eq_3.69}  that
	\begin{equation}\label{eq_3.72}
		\begin{split}
			\frac{{\rm d}\mu}{{\rm d}t} &= \langle \partial_t\Psi^+, \nabla\mu\rangle = \langle \partial_t\Psi^+, \partial_r|_{r=R(t)}\rangle
			= -n(H(R(t))-\sigma)\langle N_t, \partial_r\rangle|_{r=R(t)}\\
			&=  - n(H(R(t))-\sigma) = -n(H(r_0+\mu(t))-\sigma)
		\end{split}
	\end{equation}
	what means that $t$ coincides with the parameter defined in \eqref{eq_3.68} and then satisfying the condition that $\partial B_{R(t)}=\Psi_t^+ (\partial B_{r_0})$. Note that $R(t)\ge r_0$ for $t\ge 0$. We conclude that the one-parameter family of functions $u_+(x, t) = v_{R(t)}(r(x))$ defined on the common domain $B_{r_0}(o)$ defines a solution of the geometric flow \eqref{eq_3.69}. Hence, we set
	\begin{equation}\label{eq_3.73}
		\Psi^+(x, t) = (x, u_+(x, t)), \quad x\in B_{r_0}(o).
	\end{equation} 
	It follows that $u_+$ satisfies the parabolic equation
	\begin{equation}\label{eq_3.74}
		\begin{split}
			& \partial_t u_+ X = \partial_t \Psi^+ =  -\Bigg(\partial_r\left(\frac{\partial_r u_+}{(\varrho^{-2}_+(r)+|\partial_r u_+|^2)^{1/2}}\right)\\
			& + \frac{\partial_r u_+}{(\varrho^{-2}_+(r)+|\partial_r u_+|^2)^{1/2}}\left(\frac{\varrho'_+(r)}{\varrho_+(r)}+(n-1)\frac{\xi'_+(r)}{\xi_+(r)}\right)-n\sigma\Bigg) N(t),
		\end{split}
	\end{equation}
	that is
	\begin{equation}\label{eq_3.75}
		\begin{split}
			& \partial_t u_+ =  -(\varrho^{-2}_+(r)+|\partial_r u_+|^2)^{1/2}  \Bigg(\partial_r\left(\frac{\partial_r u_+}{(\varrho^{-2}_+(r)+|\partial_r u_+|^2)^{1/2}}\right)\\
			& + \frac{\partial_r u_+}{(\varrho^{-2}_+(r)+|\partial_r u_+|^2)^{1/2}}\left(\frac{\varrho'_+(r)}{\varrho_+(r)}+(n-1)\frac{\xi'_+(r)}{\xi_+(r)}\right)-n\sigma\Bigg).
		\end{split}
	\end{equation}

\subsection{Barriers and height estimates for the modified mean curvature flow}\label{C0-bounds}
    
	In the previous section, we have constructed a  modified mean curvature flow by constant mean curvature spherical caps in the model manifold whose Riemannian metric is defined by \eqref{eq_3.49}-\eqref{eq_3.51}.

    Besides providing a relevant example, that flow may be used as a supersolution to the modified mean curvature flow in manifolds whose metrics are not necessarily invariant by rotations. This is the content of the following propositions.
	
    \begin{proposition}\label{supersolution}
		Suppose that \eqref{cond-3} holds  with $\xi(r) = \xi_+(r)$ and that  $\varrho(x) = \varrho_+(r(x))$.
		Then, the one-parameter family of functions
		\begin{equation}\label{eq_3.76}
			u_+(x, t) = v_{R(t)} (r(x)), \quad x\in B_{r_0}(o), \quad t\in [0,\infty)
		\end{equation} 
		satisfies $\partial_t u_+ + \mathcal{Q}[u_+] \ge 0$ in $B_{r_0}(o)\times [0, T)$ for all $T>0$.
	\end{proposition}
	
	\begin{proof}
		Denoting
		\begin{equation}\label{eq_3.77}
			W = (\varrho^{-2}+|\nabla^M u_+|^2)^{1/2} = (\varrho^{-2}_+ + u'^2_+(r))^{1/2}
		\end{equation}
		we have
		\begin{equation}\label{eq_3.78}
			\begin{split}
				&  \mathcal{Q}[u_+] +\partial_t u_+= W\Bigg({\rm div}_M\left(\frac{\nabla^M u_+}{W}\right) + \left\langle\nabla^M\log\varrho, \frac{\nabla^M u_+}{W}\right\rangle-n\sigma\Bigg) +\partial_t u_+\\
				& \,\, =(\varrho^{-2}+ u'^2_+(r))^{1/2}\Bigg(\partial_r\left(\frac{u'_+(r)}{(\varrho^{-2}_+ + u'^2_+(r))^{1/2}}\right) \\
				& \,\,\,\,+ \frac{u'_+(r)}{(\varrho^{-2}_+ + u'^2_+(r))^{1/2}} \big(\Delta_M r + \langle \nabla^M \log \varrho, \nabla^M r\rangle\big) -n\sigma\Bigg)+ \partial_t u_+.
			\end{split}
		\end{equation}
		However
		\begin{equation}\label{eq_3.79}
			\langle \nabla \log \varrho, \nabla r\rangle = \frac{\partial_r \varrho}{\varrho} = \frac{\varrho'_+(r)}{\varrho_+(r)}\cdot
		\end{equation}
		Moreover \eqref{hess-comp} implies that 
		\begin{equation}\label{eq_3.80}
			\Delta_M r \le (n-1)\frac{\xi'_+(r)}{\xi_+(r)}\cdot
		\end{equation}
		Since $u'_+= v'_R\le 0$ we conclude that
		\begin{equation}\label{eq_3.81}
			\begin{split}
				&  \mathcal{Q}[u_+] +\partial_t u_+ \ge (\varrho_+^{-2}+u'^2_+(r))^{1/2}\Bigg(\partial_r\left(\frac{u'_+(r)}{(\varrho^{-2}_+ + u'^2_+(r))^{1/2}}\right)\\
				& \,\,\,\,+\frac{u'_+(r)}{(\varrho^{-2}_+ + u'^2_+(r))^{1/2}}\left(\frac{\varrho'_+(r)}{\varrho_+(r)}+(n-1)\frac{\xi'_+(r)}{\xi_+(r)}\right)-n\sigma\Bigg)   + \partial_t u_+=0.
			\end{split}
		\end{equation}
		This finishes the proof.
	\end{proof}
	\begin{proposition}
		\label{height-mcf}
		Suppose that \eqref{cond-3} holds  with $\xi(r) = \xi_+(r)$ and that  $\varrho(x) = \varrho_+(r(x))$.  Let $u$ be a solution of  \eqref{eq_3.8}-\eqref{eq_3.9} in $B_R(o)\times [0,T]$  with  Dirichlet boundary condition $u(x,t) = u(x,0)$ for any $(x, t) \in \partial B_{r_0}(o) \times [0,T]$. Then we have the following height estimate
		\begin{equation}\label{eq_3.82}
			|u(x,t)| \le {\sup}_{B_{r_0}(o)} |u| 
			+v_{R(T)} (o) - v_{r_0} (r(x)),
		\end{equation}
		that is,
		\begin{equation}\label{eq_3.83}
			\begin{split}
				& |u(x,t) |\le {\sup}_{B_{r_0}(o)} |u(\cdot, 0)| + \int^{0}_{R(T)} \frac{nH(R(T)) V(\varsigma)}{\varrho(\varsigma) (A^2(\varsigma)-n^2 H^2(R(T))V^2(\varsigma))^{\frac{1}{2}}}\,{\rm d}\varsigma\\
				& \,\, - \int^{r(x)}_{r_0} \frac{nH(r_0) V(\varsigma)}{\varrho(\varsigma) (A^2(\varsigma)-n^2 H^2(r_0)V^2(\varsigma))^{\frac{1}{2}}}\,{\rm d}\varsigma.
			\end{split}
		\end{equation}
		for $(x,t) \in B_{r_0}(o) \times [0, T]$.
	\end{proposition}

	\begin{proof} By construction, the graph $\Sigma_{r_0}$ of $u_+(\cdot, 0) =v_{r_0}$ is defined in the geodesic ball $B_{r_0}(o)$. Given $T>0$ we have that
		$\Psi^+_T(\Sigma_{r_0})$ is the graph $\Sigma_{R(T)}$ of $u_+ (\cdot, T) = v_{R_T}|_{B_{r_0}(o)}$ with
		\begin{equation}\label{eq_3.84}
			\int_{r_0}^{R(T)} \frac{V(\varsigma)}{A(\varsigma)}\, {\rm d}\varsigma = T.
		\end{equation}
		Given $\varepsilon>0$ we have
		\begin{equation}\label{eq_3.85}
			-u_+(x,T) + u_+(o,T)+ {\sup}_{B_{r_0}(o)} u + \varepsilon > u(x, 0)
		\end{equation}
		for all $x\in B_{r_0}(o)$.  We also have
		\begin{equation}\label{eq_3.86}
			v_\varepsilon(x, t) := -u_+(x, T-t) + u_+(o,T)+ {\sup}_{B_{r_0}(o)} u + \varepsilon > u(x, t)
		\end{equation}
		for all $(x, t) \in \partial B_{r_0} (o) \times [0, T]$. Proposition \ref{supersolution} implies that
		\begin{equation}\label{eq_3.87}
			\partial_t v_\varepsilon - \mathcal{Q}[v_\varepsilon] =\partial_t u_+ + \mathcal{Q}[u_+] \ge 0
		\end{equation}
		in the parabolic cylinder $B_{r_0}(o)\times (0, T)$. Then the parabolic maximum principle implies that
		\begin{equation}\label{eq_3.88}
			u(x,t) \le v (x,t) \le v(x,T)
		\end{equation}
		in $B_{r_0} (o) \times [0, T]$ where
		\begin{equation}\label{eq_3.89}
			v(x, t) = -u_+(x, T-t) + u_+(o,T)+ {\sup}_{B_{r_0}(o)} u.
		\end{equation}
		Hence,
		\begin{equation}\label{eq_3.90}
			u(x,t) \le v(x,T) = u_+(o,T)-u_+(x, 0) + {\sup}_{B_{r_0}(o)} u.
		\end{equation}
		Therefore
		\begin{equation}\label{eq_3.91}
			u(x,t) \le {\sup}_{B_{r_0}(o)} u 
			+v_{R(T)} (o) - v_{r_0} (r(x))
		\end{equation}
		for $(x,t) \in B_{r_0} (o) \times [0, T]$.  We prove in a similar way that 
		\begin{equation}\label{eq_3.92}
			u(x,t) \ge  w (x,t) \ge w(x,T)
		\end{equation}
		in $B_{r_0} (o) \times [0, T]$ where
		\begin{equation}\label{eq_3.93}
			w (x, t) = u_+(x, T-t) - u_+(o,T)+ {\inf}_{B_{r_0}(o)} u.
		\end{equation}
		Hence
		\begin{equation}\label{eq_3.94}
			u(x,t) \ge {\inf}_{B_{r_0}(o)} u -
			v_{R(T)} (o) + v_{r_0} (r(x))
		\end{equation}
		in $B_{r_0} (o) \times [0, T]$. This finishes the proof.
	\end{proof}
	
	\section{Gradient bounds}\label{C1-est}

In this section, we obtain \emph{a priori} $C^1$ bounds for the modified mean curvature flow adapting the techniques in \cite{Ecker:89} and \cite{Ecker:91} to the general setting of Riemannian warped products. The geometric functions and differential inequalities in proposition \ref{hess-r} play a central role in the proof of these gradient estimates.
     
	\begin{proposition}\label{intest} Let $o \in M$. Given $R>0$ and $T '>0$, suppose that 
		\begin{equation}\label{eq_3.109}
			|\partial_r \log\varrho|\le \frac{\xi'(r)}{\xi(r)}
		\end{equation}
		in $B_R(o)\subset M$ and that $\overline{\rm Ric}\ge -L$ for some constant $L\ge 0$ in 
		\[
		\mathcal{S}_{R. T ' } = \bigcup_{t\in [0, T ' ]}\{\widetilde\Psi(x, t): \zeta (\widetilde\Psi(x,t)) < \bar \zeta(R)\}
		\]
        If $R'<R$, then for any $\widetilde\Psi(x,t) \in \mathcal{S}_{R,T'}$ we have that
        \begin{equation}
            W(x,t)\leq \Lambda(L,T',R)\, \max\left\{{\sup}_{\widetilde\Psi_0(B_R(o))}  W(\cdot, 0), \frac{\beta C_R}{2\, {\sup}_{B_R(o)}\xi} \, \frac{{\sup}_{B_R(o)}\varrho}{{\inf}_{B_R(o)}\varrho} \right\}.
        \end{equation}
		where $\Lambda(L,T',R)=\frac{\exp(LT')\exp(\lambda \zeta(R))}{(\exp(\lambda(\zeta(R) - \zeta(R'))) -1)}.$ In particular, 
		\begin{equation}
			\label{eq_3.110}
			W(o,t) \le \text{\small $\frac{ \exp(\lambda \zeta(R))}{\exp(\lambda \zeta(R)/2 )-1}\, \max\left\{{\sup}_{\widetilde\Psi_0(B_R(o))}  W(\cdot, 0), \frac{\beta C_R}{2\, {\sup}_{B_R(o)}\xi} \, \frac{{\sup}_{B_R(o)}\varrho}{{\inf}_{B_R(o)}\varrho} \right\}$}
		\end{equation}
		for $0\leq t\leq T ' $, where 
		\begin{equation}\label{eq_3.111}
			C_R = n\, {\sup}_{B_R(o)} (\xi'+|\sigma| \xi).
		\end{equation}
		and
		\begin{equation}\label{eq_3.112}
			\begin{split}
				\lambda =  8  \beta^2 C_R   \, {\sup}^2_{B_R(o)}\varrho  
			\end{split}
		\end{equation}
		with
		\begin{equation}\label{eq_3.113}
			\beta:  = \frac{1}{\bar\zeta(R)}\, \big({\sup}_{ \mathcal{S}_{R, T '} }  |s| - {\sup}_{[0, T ' ]} (s\circ\widetilde \Psi) (o,t)\big).
		\end{equation}
		and $\alpha = {\sup}_{B_R(o)} \left(2|\bar\nabla\log\varrho|+n|\sigma|\right)$ and  $\widetilde \delta =2\, {\sup}_{B_R(o)}(2|\bar\nabla\log\varrho|+2\xi+n|\sigma|)$. 
	\end{proposition}

	\begin{proof}  Denote
		\begin{equation}\label{eq_3.114}
			\alpha = {\sup}_{B_R(o)} \left(2|\bar\nabla\log\varrho|+n|\sigma|\right).
		\end{equation}
		Following \cite{Korevaar:86} and  \cite{Ecker:91} we define
		\begin{equation}\label{eq_3.115}
			\phi(\tau) = e^{\lambda \tau} -1, \quad \tau\ge 0.
		\end{equation}
		Given the function  $\zeta$  defined in \eqref{eq_3.12}  one sets 
		\begin{equation}\label{eq_3.116}
			\eta(\widetilde\Psi(x,t)) = e^{-Lt} \phi\left(\left(\bar\zeta(R)-\zeta(\widetilde \Psi(x,t))+\chi(s(\widetilde \Psi(x,t)))\right)_+\right),
		\end{equation}
		for $(x,t)\in M\times [0, T']$ such that $\zeta(\widetilde\Psi(x,t)) < \bar\zeta(R)$.
		The function $\chi$ in \eqref{eq_3.116} is defined by
		\begin{equation}\label{eq_3.117}
			\chi(s) = \frac{1}{2\beta} \left(s -{\sup}_{\mathcal{S}_{R, T '}} |s\circ\widetilde\Psi|\right).
		\end{equation}
		where
		\begin{equation}\label{eq_3.118}
			\mathcal{S}_{R, T'} = {\bigcup}_{t\in [0, T ']} S_{R, t}, \,\, \mbox{ with } \,\, S_{R, t} = \{\widetilde \Psi(x, t): \zeta (\widetilde \Psi(x,t))  < \bar\zeta (R)\}.
		\end{equation}
		Using \eqref{eq_3.27} and denoting $\phi\circ\widetilde\Psi$ simply as $\phi$ one computes
		\begin{equation}\label{eq_3.119}
			\begin{split}
				 \left(\partial_t-\Delta\right)(\eta W)&=W\left(\partial_t-\Delta\right)\eta+\eta\left(\partial_t-\Delta\right)W-2\left\langle\nabla W,\nabla\eta\right\rangle\\
				&=W\left(\partial_t-\Delta\right)\eta-(|A|^2+\overline{\Ric}(N,N))\eta W 
                \\
                &\leq  e^{-Lt}W\left(\partial_t-\Delta\right)\phi.
			\end{split}
		\end{equation}
		However
		\begin{equation}\label{eq_3.120}
			\begin{split}
				& (\partial_t- \Delta)\phi= \left((\partial_t-\Delta)\zeta + \dot \chi(s) (\partial_t-\Delta) s -  \ddot \chi(s) |\nabla s|^2 \right) \phi' \\
				& \,\,-  \left(|\nabla \zeta|^2 + \dot \chi^2(s) |\nabla s|^2 - 2\xi(r) \dot\chi(s) \langle \nabla r, \nabla s\rangle\right)\phi''
			\end{split}
		\end{equation}
		Since $\ddot \chi=0$,  $\phi'>0$ and using \eqref{eq_3.11}, \eqref{eq_3.13} and  \eqref{eq_3.109} one gets
		\begin{equation}\label{eq_3.121}
			\begin{split}
				& (\partial_t - \Delta) \phi \le \phi' \big( n\xi'(r) + n \sigma \xi(r) \langle\bar\nabla r, N\rangle -  \dot \chi(s) (2\langle\bar{\nabla}\log\varrho,N\rangle+n\sigma)\langle\bar{\nabla}s,N\rangle\big)  \\
				& \,\, -  \phi'' \left(|\nabla \zeta|^2 + \dot \chi^2(s) |\nabla s|^2 + 2\xi(r) \dot\chi(s) \langle \bar\nabla r, N\rangle \langle  \bar\nabla s, N\rangle\right)
			\end{split}
		\end{equation} 
		where we used the fact that $\langle \nabla r, \nabla s\rangle = -\langle\bar\nabla r, N\rangle \langle \bar\nabla s, N\rangle$. 
		Note that 
		\begin{equation}\label{eq_3.122}
			n\xi'(r) + n \sigma \xi(r) \langle\bar\nabla r, N\rangle  \le C_R=:\delta.
		\end{equation}
		Since $\phi''\ge 0$ one obtains rearranging terms and discarding a non-negative term that 
		\begin{equation}\label{eq_3.123}
			\begin{split}
				& (\partial_t - \Delta) \phi \le  \delta \phi'+\left(\alpha \phi' -2\xi(r)\langle \bar\nabla r, N\rangle\phi''\right) \dot \chi(s) \langle\bar{\nabla}s,N\rangle  - \dot\chi^2 (s)\phi'' |\nabla s|^2 
			\end{split}
		\end{equation} 
		Note that
		\begin{equation}\label{eq_3.124}
			\langle \bar\nabla s , N\rangle = \frac{1}{\varrho^2}\langle X, N\rangle = \frac{1}{\varrho^2}\frac{1}{W}
		\end{equation}
		and
		\begin{equation}\label{eq_3.125}
			|\nabla s|^2 = |\bar\nabla s|^2 -\langle \bar\nabla s, N\rangle^2 = \frac{1}{\varrho^2}-\frac{1}{\varrho^4}\langle X, N\rangle^2 = 
		\frac{1}{\varrho^2}\left(1-\frac{1}{\varrho^2}\frac{1}{W^2}\right).
		\end{equation}
		Therefore
		\begin{equation}\label{eq_3.126}
			\begin{split}
				& (\partial_t - \Delta) \phi \le  \delta \phi'+ \left(\alpha \phi' + 2\xi(r)\phi''\right) \frac{\dot \chi}{\varrho} \frac{1}{\varrho W}  - \frac{\dot\chi^2 }{\varrho^2}\phi''\left(1-\frac{1}{\varrho^2}\frac{1}{W^2}\right).
			\end{split}
		\end{equation} 
		Rearranging terms, one obtains 
		\begin{equation}\label{eq_3.127}
			\begin{split}
				\left(\partial_t-\Delta\right)\phi \le  \phi' \left(\delta+ \alpha \frac{\dot \chi}{\varrho} \frac{1}{\varrho W}\right)+\phi'' \left(2\xi \frac{\dot \chi}{\varrho} \frac{1}{\varrho W} - \frac{\dot\chi^2 }{\varrho^2}\left(1-\frac{1}{\varrho^2}\frac{1}{W^2}\right)\right).  
			\end{split}
		\end{equation}
		Since $\phi'= \lambda (\phi+1)$, $\phi''  = \lambda^2 (\phi+1)$  and $\dot\chi = 1/2\beta$ we conclude from \eqref{eq_3.118} that in an (interior maximum point) of $\eta W$ one has
		\begin{equation}\label{eq_3.128}
			\lambda \left(\frac{1}{4\beta^2} \frac{1}{\varrho^2} \left(1-\frac{1}{\varrho^2}\frac{1}{W^2}\right) -  2\xi \frac{1}{2\beta}\frac{1}{\varrho} \frac{1}{\varrho W} \right) \le C_R + \alpha \frac{1}{2\beta}\frac{1}{\varrho} \frac{1}{\varrho W}\cdot
		\end{equation}
		Fixing $\lambda$ as in \eqref{eq_3.111} yields
		\begin{equation}
			C_R \varrho^2 W^2  - \bigg(4\beta\,{\rm sup}_{B_R(o)}\xi \frac{{\rm sup}^2_{B_R(o)} \varrho}{{\rm inf}_{B_R(o)}\varrho} +  \frac{\alpha}{2\beta{\rm \inf}_{B_R(o)}\varrho}\bigg) \varrho W
			- 2C_R \le 0
		\end{equation}
		{Hence, solving the quadratic inequality, one obtains}
		\begin{equation}\label{eq_3.129}
			\begin{split}
				&  W(x_0, t_0) \le  \frac{2\beta}{C_R}\,{\rm sup}_{B_R(o)}\xi \frac{{\rm sup}^2_{B_R(o)} \varrho}{{\rm inf}^2_{B_R(o)}\varrho} +  \frac{\alpha}{4\beta C_R \, {\rm \inf}^2_{B_R(o)}\varrho} \\
				& \,\,+ \bigg(2+\bigg( \frac{2\beta}{C_R}\,{\rm sup}_{B_R(o)}\xi \frac{{\rm sup}^2_{B_R(o)} \varrho}{{\rm inf}^2_{B_R(o)}\varrho} +  \frac{\alpha}{4\beta C_R \, {\rm \inf}^2_{B_R(o)}\varrho}\bigg)^2\bigg)^{\frac{1}{2}} =: \mathcal{C}_0 (R).
			\end{split}
		\end{equation}
		at this maximum point. Therefore
		\begin{equation}\label{eq_3.130}
        \begin{split}
			{\rm sup}_{\widetilde\Psi_{t_0} (M)} \eta W &\le \eta (\widetilde\Psi(x_0, t_0)) W(x_0, t_0) \le e^{-Lt_0} e^{\lambda \bar\zeta(R)}W(x_0, t_0) 
            \\
            &\le e^{-Lt_0} e^{\lambda \bar\zeta(R)}\mathcal{C}_0 (R).
        \end{split}
		\end{equation}
		Since the only two possibilities are that either the maximum is attained in $t=0$ or that it is attained in some leaf $\widetilde\Psi_{t_0} (M)$ for some $t_0>0$ (the function vanishes in 
		the parabolic boundary, since $\eta=0$ wherever $\zeta(\widetilde\Psi(x,t)) = \zeta(R)$ for $t\in [0, T' ]$), we conclude that
		\begin{equation}\label{eq_3.131}
			{\sup}_{\mathcal{S}_{R, T '}} \eta W \le {\sup}_{\widetilde\Psi (B_R(o), 0)} \eta W(\cdot, 0) + e^{-Lt_0} e^{\lambda \bar\zeta(R)} \mathcal{C}_0 (R).
		\end{equation}
		{Moreover, since $\xi$ is positive, we have that $\overline{\zeta}$ is increasing. Hence, if $\mathbf{y}\in\overline{\mathcal{S}_{R',T'}}$ then there exists a sequence $\widetilde\Psi(x_n,t_n) \in {\mathcal{S}_{R',T'}}$ such that $\mathbf{y}=\lim\limits_{n\to \infty} \widetilde\Psi(x_n,t_n)$ and 
        \begin{equation}
            \zeta(\mathbf{y})=\lim\limits_{n\to \infty}\zeta(\widetilde\Psi(x_n,t_n)) \leq \overline{\zeta}(R') <\overline{\zeta}(R).
        \end{equation}
        Hence, $\overline{\mathcal{S}_{R',T'}}\subset \mathcal{S}_{R,T'}$ and $\sup_{\overline{\mathcal{S}_{R',T'}}} \eta W\leq \sup_{\mathcal{S}_{R,T'}} \eta W$. We conclude that
        \begin{equation}
            \text{sup}_{\overline{\mathcal{S}_{R',T'}}} \eta\cdot \text{sup}_{\overline{\mathcal{S}_{R',T'}}} W \le  {\sup}_{\widetilde\Psi (B_R(o), 0)} \eta W(\cdot, 0) + e^{-Lt_0} e^{\lambda \bar\zeta(R)} \mathcal{C}_0 (R).
        \end{equation}}
In particular
		\begin{equation}\label{eq_3.132}
			\eta(o,t ) W(o, t) \le e^{\lambda \bar\zeta(R)} \big({\sup}_{\widetilde\Psi (B_R(o), 0)}  W(\cdot, 0) + e^{-Lt_0}  \mathcal{C}_0 (R)\big).
		\end{equation}
		We conclude that
		\begin{equation}\label{eq_3.133}
			\begin{split}
				& W(o,t) \le  \frac{e^{\lambda \bar\zeta(R)}}{e^{\lambda\bar\zeta(R)/2}-1}  e^{LT}\, \big({\sup}_{\widetilde\Psi (B_R(o), 0)}  W(\cdot, 0) + e^{-Lt_0}  \mathcal{C}_0 (R)\big)\\
			\end{split}
		\end{equation}
        This finishes the proof.
	\end{proof}
	The following proposition establishes boundary gradient estimates for the modified mean curvature flow \eqref{MCF-non-par}. Jointly with the height and interior gradient estimates this result allows us to prove the longtime existence of \eqref{MCF-non-par} over geodesic balls in $M$ centered at $o$.
    
        \begin{proposition}\label{prop_4.1}
		Let $u$ be a solution of \eqref{eq_3.8}-\eqref{eq_3.9} in
        $B_R(o) \times[0, T)$ for $R>0$ and $T>0$ with Dirichlet boundary condition $u(x, t) = u_0(x)$ for $(x, t) \in\partial B_R(o) \times [0, T)$. Then there exists a constant $C>0$ such that
		\begin{equation}\label{eq_4.1}
			\sup_{\partial B_R(o) \times[0, T]}|\nabla u| \leq C .
		\end{equation}
	\end{proposition}
	\noindent \emph{Proof.} In order to estimate the gradient of $u(\cdot, t)$ along the boundary $\partial B_R(o)$ we consider a function of the form
	\begin{equation}\label{eq_4.2}
		v(x) =\widetilde u_0 (x) +  h(d(x))
	\end{equation}
	where $d(x) = R-r(x)$ for $x\in B_R(o)$ and $\widetilde u_0$ is a local extension of $u_0$ defined for $d<\varepsilon$ for some $\varepsilon>0$ such that the there are no focal points of $\partial B_R(o)$ for $R-\varepsilon< r < R+\varepsilon$. In what follows we choose a  function $h$ in such a way that $v$ is a supersolution of \eqref{eq_3.8} defined in the neighborhood of $\partial B_R (o)$ given by the points whose geodesic distance to $\partial B_R(o)$ is less than $\varepsilon$. The following calculations are done in that neighborhood. 
	Denoting 
	\begin{equation}\label{eq_4.3}
		W = \sqrt{\varrho^{-2} + |\nabla^M v|^2} = \sqrt{\varrho^{-2}+ h'^2(d) + 2h'(d)\langle\nabla^M d, \nabla^M \widetilde u_0\rangle + |\nabla^M \widetilde u_0|^2 } 
	\end{equation}
	one computes
	\begin{equation}\label{eq_4.4}
		\begin{split}
			& \partial_t v - \mathcal{Q}[v] = \partial_t v - \Delta_M v + \frac{1}{W^2}\langle \nabla^M_{\nabla^M v}\nabla^M v, \nabla^M v\rangle 
            \\
            &-\left(1+\frac{1}{\varrho^2 W^2}\right)\langle \nabla^M \log\varrho, \nabla^M v\rangle+ n\sigma W.
		\end{split}
	\end{equation}
	Using that $|\nabla^M d|=1$ and the definition of $v$, we have
		\begin{align}
		\label{eq_4.8}
				& W^2 (\partial_t v - \mathcal{Q}[v])=- {\varrho^{-2}h''(d)}-{\left(\varrho^{-2}+(h'(d))^2\right)h'(d)\Delta_M d}
				\nonumber\\
				& \,\, - {\left(W^2+\varrho^{-2}\right)\langle \nabla^M \log\varrho, \nabla^M \widetilde{u}_0 + h'(d)\nabla^M d \rangle} \nonumber\\
		& 	\,\,	- (h'(d))^2 \left({\Delta_M \widetilde u_0} + {2| \nabla^M \widetilde{u}_0| |\Delta_M d|}+{| \nabla^M \nabla^M \widetilde{u}_0|} \right)
				\nonumber\\
				&\,\, -h'(d)\left({|\nabla^M \widetilde{u}_0|^2 |\Delta_M d|}+{| \nabla^M \nabla^M  d| |\nabla^M \widetilde{u}_0|^2}\right)
                \\
                &\,\, -h'(d)\left({2 |\nabla^M \widetilde{u}_0 | |\Delta_M\widetilde{u}_0|} + {2| \nabla^M  \nabla^M \widetilde{u}_0| |\nabla^M \widetilde{u}_0|}\right) \nonumber
				\\
				&\,\, - \left(\varrho^{-2}+|\nabla^M \widetilde{u}_0|^2\right){\Delta_M \widetilde{u}_0} - {|\nabla^M \widetilde{u}_0|^2|\nabla^M \nabla^M \widetilde{u}_0|} +{n\sigma W^3}.
		\end{align}
	Note that $|\Delta_M \widetilde{u}_0| \leq \sqrt{n}|\nabla^M \nabla^M \widetilde{u}_0|.$	Hence,
	\begin{align}\label{eq_4.10}
			& W^2 (\partial_t v - \mathcal{Q}[v]) \geq - {\varrho^{-2}h''(d)}-{\left(\varrho^{-2}+(h'(d))^2\right)h'(d)\Delta_M d}
			\nonumber\\
			& \,\, - {\left(W^2+\varrho^{-2}\right)\langle \nabla^M \log\varrho, \nabla^M \widetilde{u}_0 + h'(d)\nabla^M d \rangle}
			\nonumber\\
			&\,\,- C(h'(d))^2 \left(|\nabla^M \widetilde{u}_0 +{| \nabla^M \nabla^M \widetilde{u}_0|} \right)
			\nonumber\\
			&\,\, -Ch'(d)\left(|\nabla^M \widetilde{u}_0|^2 + {| \nabla^M  \nabla^M \widetilde{u}_0| |\nabla^M \widetilde{u}_0|}\right)
			\nonumber\\
			&\,\, - C\left(\varrho^{-2}+|\nabla^M \widetilde{u}_0|^2\right)|\nabla^M \nabla^M \widetilde{u}_0| + {n\sigma W^3}.
	\end{align}
	where $C=2\mathrm{max}(\sqrt{n}+1,\sqrt{n}|\nabla^M \nabla^M d|).${ 
	Moreover, by equations \eqref{eq_3.79} and \eqref{eq_3.80}, and the definition of the function $d$, we have
	\begin{equation}\label{eq_4.11}
		-\left(\Delta_M d + \langle \nabla^M \log \varrho, \nabla^M d\rangle\right) =\Delta_M r+ \langle \nabla^M \log \varrho, \nabla^M r\rangle \geq -nB
	\end{equation}
	where $B=\mathrm{sup}_{(R-\varepsilon, R+\varepsilon)} \frac{\xi'(r)}{\xi(r)}$. Hence, change the constant $C$ contain the negative terms if necessary, we have
	\begin{equation}\label{eq_4.12}
		\begin{split}
			& W^2 (\partial_t v - \mathcal{Q}[v]) \geq - {\varrho^{-2}h''(d)}-h'(d)|\nabla^M \log\varrho|\varrho^{-2}-nB(h'(d))^3
			\\
			&\,\,- C(h'(d))^2 \left(|\nabla^M \log\varrho||\nabla^M \widetilde{u}_0|+|\nabla^M \widetilde{u}_0| +{| \nabla^M \nabla^M \widetilde{u}_0|} \right)
			\\
			&\,\, -Ch'(d)\left(nB\varrho^{-2}+|\nabla^M \log\varrho||\nabla^M \widetilde{u}_0|^2+|\nabla^M \widetilde{u}_0|^2 + {| \nabla^M  \nabla^M \widetilde{u}_0| |\nabla^M \widetilde{u}_0|}\right)
			\\
			&\,\, - C\left(\varrho^{-2}+|\nabla^M \widetilde{u}_0|^2\right)\left(|\nabla^M \log \varrho||\nabla^M \widetilde{u}_0|+|\nabla^M \nabla^M \widetilde{u}_0|\right) + {n\sigma W^3}.
		\end{split}
	\end{equation}
	For $L>0$, fix $d_0<L^{-1}$ and $A=L(1-Ld_0)^{-1}$. If we take $h(d)=L^{-1}\log(1+Ad)$, then 
	\begin{equation}\label{eq_4.13}
		h'(d)=\dfrac{1}{L}\left(\dfrac{A}{1+Ad}\right),\, \, \, \, h''(d)=-L(h'(d))^2
	\end{equation}
	Thence, for all $d<d_0$,
	\begin{equation}\label{eq_4.14}
		\begin{split}
				W^2\left(\partial_t v-\mathcal{Q}[v]\right) & \geq \frac{1}{L}\frac{A}{1+Ad}\left[-nB \frac{1}{L^2}\left(\frac{A^2}{(1+Ad)^2}\right)+\left(\varrho^{-2}-\frac{\overline{C}}{L}\right)\frac{A}{1+Ad}\right. \\
				& -\left.\left(|\nabla^M \log \varrho|\varrho^{-2} + \overline{C}+ \overline{C}L\left(\frac{1+Ad}{A}\right)\right)\right]
		\end{split}
	\end{equation}
	where $\overline{C}=\overline{C}\left(n,B,\varrho,|\nabla^M \log\varrho|, |\nabla^M \widetilde{u}_0|,|\nabla^M\nabla^M \widetilde{u}_0| \right).$ Since 
	\begin{equation}\label{eq_4.15}
		\begin{split}
			-\left(\frac{1+Ad}{A}L\overline{C}\right)\geq -\overline{C},
		\end{split}
	\end{equation}
	we have the following inequality
	\begin{equation}\label{eq_4.16}
	\begin{split}
		\text{\small $W^2\left(\partial_t v-\mathcal{Q}[v]\right) \geq \frac{\overline{C}}{L}\frac{A}{1+Ad}\left[\left(\frac{L}{\varrho^{2}\overline{C}}-1\right)\frac{1}{L}\left(\frac{A}{1+Ad}\right)- \frac{1}{L^2}\left(\frac{A^2}{(1+Ad)^2}\right)-3\right]$.}
	\end{split}
\end{equation}
	Note that
	\begin{equation}\label{eq_4.17}
		\left(\frac{L}{\varrho^2\overline{C}}-1\right)^2-12 > 0
	\end{equation}
	for a chosen $L>5\overline{C}\mathrm{sup}_{B_R(o)} \varrho^{-2}$. Note that the roots of the polynomial 
	\begin{equation}\label{eq_4.18}
		\mathfrak{p}(x)=-x^2+4Lx^2-3L^2
	\end{equation}
	are $L$ and $3L$. Moreover, $\mathfrak{p}(x)\geq 0$ for $x\in [L,3L]$. If $x=A(1+Ad)^{-1}$, then $x\geq L$. If we choose $d_0 \geq 1-3L$, then $x\leq 3L$. We conclude that
	\begin{equation}\label{eq_4.19}
		\partial_t v - \mathcal{Q}[v] \geq 0.
	\end{equation}
    This finishes the proof. \hfill $\square$
\medskip

    Since \eqref{eq_6.2} may be written in local coordinates as a parabolic equation, the \emph{a priori} $C^0$ and $C^1$ bounds in propositions \ref{height-mcf}, \ref{intest} and \ref{prop_4.1} yield the following local existence result.

    \begin{theorem}\label{local-thm}
        Let $M$ be a $n$-dimensional complete, non-compact oriented  Riemannian manifold with a pole $o$ and let  $\bar{M}$ be the warped product $M\times_\varrho\mathbb{R}$ for some positive function $\varrho\in C^\infty(M)$. Suppose that {\eqref{cond-1}}  and {\eqref{cond-3}} hold.  Given any $R>0$ there exists a smooth solution $u: \overline{B_R(o)} \times [0, \infty) \to \mathbb{R}$ of \eqref{eq_3.8}-\eqref{eq_3.9} with Dirichlet boundary condition $u(x, 0) = u_0(x)$ for $x\in \partial B_R(o)$ and initial condition $u(\cdot, 0) = u_0$ 
        for any $\sigma$ satisfying
	\begin{equation}\label{eq_2.13.a}
		\sigma < \frac{1}{n}\inf \left( \frac{|\bar\nabla  \varrho|}{\varrho}+(n-1) \frac{\xi'(r)}{\xi(r)}\right).
        \end{equation}
    \end{theorem}

    \section{Curvature estimates}\label{C2-est}

	In order to obtain second order bounds we need to deduce evolution equations for the second fundamental form and its squared norm, a parabolic counterpart  of the classical Simons' formula.

	\begin{lemma} \label{evol|A|}
		The squared norm $|A|^2$ of the second fundamental form of  $\Sigma_t$, $t\in [0,T]$, evolve as
		\begin{equation}
			\label{eq_5.1}
			\begin{split}
				&  \frac{1}{2} (\partial_t-\Delta) |A|^2 + |\nabla A|^2  =-n\sigma a_i^sa_{sj} a^{ij} +|A|^4 +n(H-\sigma)a^{ij}\bar{R}_{i00j}\\
				& \,\,+g^{k\ell}\left(\nabla_iL_{kj\ell}+\nabla_kL_{\ell ij}\right)a^{ij}+g^{k\ell} (a_{is}  \bar R^s_{kj\ell} +  a_{sk}  \bar R^s_{\ell ij}) a^{ij}
			\end{split}
		\end{equation}
		where $L$ is the $(0,3)$-tensor  in $\Sigma_t$ defined by $L_{ijk}=\langle\bar{R}(\partial_i,\partial_j)N,\partial_k\rangle$. This expression is rewritten in terms of the ambient curvature tensor as
		\begin{equation}
			\label{eq_5.2}
			\begin{split}
				& \frac{1}{2} (\partial_t-\Delta) |A|^2 + |\nabla A|^2  =-n\sigma ( a_i^sa_{sj} +\bar{R}_{i00j} ) a^{ij} +|A|^4  + |A|^2 \overline{{\rm Ric}}(N,N)
				\\
				& \,\,+g^{k\ell}(\bar\nabla_{i} \bar R_{kj0\ell} + \bar\nabla_k \bar R_{\ell i0j})a^{ij}  + 2  g^{k\ell} (a_{is}  \bar R^s_{kj\ell} +  a_{sk}  \bar R^s_{\ell ij}) a^{ij}.
			\end{split}
		\end{equation}
	\end{lemma}

	\begin{proof} We have
		\begin{equation}\label{eq_5.3}
			\partial_ta_{ij}= n \nabla_i \nabla_j H - n(H-\sigma) a_{is} a^s_j+n(H-\sigma)\bar{R}_{i00j}
		\end{equation}
		Since
		\begin{equation}\label{eq_5.4}
			\partial_tg^{ij}=2n(H-\sigma)a^{ij}
		\end{equation}
		we have
		\begin{equation}\label{eq_5.5}
			\begin{split}
				& \frac{1}{2}\partial_t|A|^2 = g^{j\ell} a_{ij} a_{k\ell}  \partial_t g^{ik} + 
				g^{ik} g^{j\ell} a_{k\ell}\partial_t a_{ij} = 2n(H-\sigma)a^{ik} a_{i}^\ell a_{k\ell}  \\
				& \,\, +  a^{ij}  (n \nabla_i \nabla_j H - n(H-\sigma) a_{i\ell} a^\ell_j+n(H-\sigma)\bar{R}_{i00j}).
			\end{split}
		\end{equation}
		We conclude that
		\begin{equation}\label{eq_5.6}
			\frac{1}{2}\partial_t|A|^2 = n(H-\sigma)a^{ik} a_{i}^\ell a_{k\ell}  +    n a^{ij} \nabla_i \nabla_j H +n(H-\sigma)a^{ij}\bar{R}_{i00j}.
		\end{equation}
		On the other hand
		\begin{equation}\label{eq_5.7}
        \begin{split}
			\Delta a_{ij}&=n\nabla_i\nabla_jH+nHa_i^sa_{sj}-a_{ij}|A|^2-g^{k\ell}\left(\nabla_iL_{kj\ell}+\nabla_kL_{\ell ij}\right)
            \\
            &+g^{k\ell}(\bar{R}^s_{ik\ell }a_{sj}+\bar{R}^s_{ikj}a_{\ell s})
        \end{split}
		\end{equation}
		and
		\begin{equation}
			\label{eq_5.8}
			\begin{split}
				& \frac{1}{2}\Delta |A|^2 - |\nabla A|^2 = a^{ij} \Delta a_{ij} =n a^{ij}\nabla_i\nabla_jH+nHa_i^sa_{sj} a^{ij}-|A|^4\\
				& \,\,-g^{k\ell}\left(\nabla_iL_{kj\ell}+\nabla_kL_{\ell ij}\right)a^{ij}+g^{k\ell}(\bar{R}^s_{ik\ell }a_{sj}+\bar{R}^s_{ikj}a_{\ell s}) a^{ij}.
			\end{split}
		\end{equation}
		Therefore
		\begin{equation}
			\label{eq_5.9}
			\begin{split}
				& \frac{1}{2} (\partial_t-\Delta) |A|^2 + |\nabla A|^2  =-n\sigma a_i^sa_{sj} a^{ij} +|A|^4 +n(H-\sigma)a^{ij}\bar{R}_{i00j}\\
				& \,\,+g^{k\ell}\left(\nabla_iL_{kj\ell}+\nabla_kL_{\ell ij}\right)a^{ij}-g^{k\ell}(\bar{R}^s_{ik\ell }a_{sj}+\bar{R}^s_{ikj}a_{\ell s}) a^{ij}.
			\end{split}
		\end{equation}
		It is worth to point out that
		\begin{equation}\label{eq_5.10}
			\begin{split}
			& \nabla_i L_{kj\ell} + \nabla_k L_{\ell ij} = \bar\nabla_{i} \bar R_{kj0\ell} + \bar\nabla_k \bar R_{\ell i0j} + a_{ik} \bar R_{0j0\ell} + a_{ij} \bar R_{k00\ell} + a_{is} \bar R^s_{kj\ell}\\
			& \,\, + a_{k\ell} \bar R_{0i0j} + a_{ki} \bar R_{\ell00 j} + a_{ks} \bar R^s_{\ell ij}.
			\end{split}
		\end{equation}
		Hence,
		\begin{equation}\label{eq_5.11}
			\begin{split}
			&g^{k\ell}(\nabla_i L_{kj\ell} + \nabla_k L_{\ell ij})a^{ij} =
            \\
            &g^{k\ell}(\bar\nabla_{i} \bar R_{kj0\ell} + \bar\nabla_k \bar R_{\ell i0j})a^{ij} - a_{i}^\ell a^{ij} \bar R_{j00\ell} + |A|^2 \overline{{\rm Ric}}(N,N) \\
			&\quad + a_{is} a^{ij} g^{k\ell} \bar R^s_{kj\ell}- nH a^{ij} \bar R_{i00j} + a_{i}^\ell a^{ij} \bar R_{\ell00 j} + g^{k\ell} a^{ij} a_{sk}  \bar R^s_{\ell ij}.
			\end{split}
		\end{equation}
		Cancelling and grouping some terms one gets
		\begin{equation}\label{eq_5.12}
			\begin{split}
			&g^{k\ell}(\nabla_i L_{kj\ell} + \nabla_k L_{\ell ij})a^{ij} = g^{k\ell}(\bar\nabla_{i} \bar R_{kj0\ell} + \bar\nabla_k \bar R_{\ell i0j})a^{ij}  + |A|^2 \overline{{\rm Ric}}(N,N) \\
			&\,\,  + a_{is} a^{ij} g^{k\ell} \bar R^s_{kj\ell}- nH a^{ij} \bar R_{i00j}  + g^{k\ell} a^{ij} a_{sk} \bar R^s_{\ell ij}.
			\end{split}
		\end{equation}
		Since
		\begin{equation}\label{eq_5.13}
			-g^{k\ell}(\bar{R}^s_{ik\ell }a_{sj}+\bar{R}^s_{ikj}a_{\ell s}) a^{ij} =  a^{ij} g^{k\ell} (a_{is}  \bar R^s_{kj\ell} +  a_{sk}  \bar R^s_{\ell ij})
		\end{equation}
		we conclude that
		\begin{equation}
			\label{eq_5.14}
			\begin{split}
				& \frac{1}{2} (\partial_t-\Delta) |A|^2 + |\nabla A|^2  =-n\sigma ( a_i^sa_{sj} +\bar{R}_{i00j} ) a^{ij} +|A|^4  + |A|^2 \overline{{\rm Ric}}(N,N)
				\\
				& \,\,+g^{k\ell}(\bar\nabla_{i} \bar R_{kj0\ell} + \bar\nabla_k \bar R_{\ell i0j})a^{ij}  + 2  g^{k\ell} (a_{is}  \bar R^s_{kj\ell} +  a_{sk}  \bar R^s_{\ell ij}) a^{ij}.
			\end{split}
		\end{equation}
		This finishes the proof.
	\end{proof}

	\vspace{3mm}
	
	\noindent Given $R>0$ and $T' \in (0, T)$ we are going to estimate $|A|$ in  the set
	\begin{equation}\label{eq_5.15}
		\mathcal{U}_{R, T'} = {\bigcup}_{t\in [0, T ' ]} U_{R,t}
	\end{equation}
	with
	\begin{equation}\label{eq_5.16}
		U_{R,t} = \big\{ y = \widetilde\Psi(x,t) : \zeta(\widetilde\Psi(x,t)) + C_R t \le \zeta(R)\big\}, \quad t\in [0, T ']. 
	\end{equation}
	In order to do this, we suppose that $W>0$ in $U_{R,t}$ and 
	we  will proceed as in \cite{Borisenko:12} studying the evolution of  the function
	\begin{equation}\label{eq_5.17}
		f=\psi(W)|A|^2,
	\end{equation}
	where 
	\begin{equation}\label{eq_5.18}
		\psi(W)=\frac{W^2}{\gamma-\delta W^2}
	\end{equation}
	with 
	\begin{equation}\label{eq_5.19}
	\gamma = \frac{1}{{\sup}_{B_{R}(o)} \varrho^2}
	\end{equation}
	and
	\begin{equation}\label{eq_5.20}
	\delta = \frac{1}{2}\frac{\gamma}{{\sup}_{B_{R'(o)}\times [0, T] } W^2 }
	\end{equation}
	for $R' \in (0, R)$ such that 
	\begin{equation}
		\label{eq_5.21}
		\zeta(r) + C_R T ' \le \zeta(R)/2 
	\end{equation}
	for $r<R'$. Therefore
	\begin{equation}\label{eq_5.22}
		\delta \psi(W) \le \frac{\gamma/2}{\gamma-\gamma/2} = 1.
	\end{equation}
	Since $\delta\psi$ is non-decreasing and $W^2\ge \varrho^{-2}$ we have
	\begin{equation}\label{eq_5.23}
		\delta \psi(W) =\frac{\delta W^2}{\gamma-\delta W^2} \ge \frac{\delta/\varrho^2}{\gamma-\delta/\varrho^2} \ge \frac{\delta}{\gamma\, {\sup}_{\mathcal{C}_R}\varrho^2 -\delta} = \frac{\delta}{1-\delta}= :\widetilde\delta.
	\end{equation}
	We also have
	\begin{equation}\label{eq_5.24}
		-\frac{2}{W}\frac{1}{\psi'(W)}-\frac{\psi''(W)}{\psi'^2(W)}+\frac{3}{2}\frac{1}{\psi(W)}<0. 
	\end{equation}
	In fact, it holds that
	\begin{equation}\label{eq_5.25}
		\begin{split}
			&  -\frac{2}{W}\frac{1}{\psi'(W)}-\frac{\psi''(W)}{\psi'^2(W)} = - \frac{(\gamma-\delta W^2)^2}{\gamma W^2}-\frac{(2\gamma^2+6\gamma\delta W^2)}{(\gamma-\delta W^2)^3}\frac{(\gamma-\delta W^2)^4}{4\gamma^2 W^2}\\
			&\,\, = - \frac{\gamma-\delta W^2}{\gamma W^2}\left(\gamma-\delta W^2+ \frac{2\gamma^2+6\gamma\delta W^2}{4\gamma}\right) =  - \frac{\gamma-\delta W^2}{\gamma W^2}\left(\frac{3}{2}\gamma+\frac{1}{2}\delta W^2\right).
		\end{split}
	\end{equation}
	Therefore
	\begin{equation}\label{eq_5.26}
		\begin{split}
		&\text{\small $-\frac{2}{W}\frac{1}{\psi'(W)}-\frac{\psi''(W)}{\psi'^2(W)} +\frac{3}{2}\frac{1}{\psi(W)}$} =    - \frac{\gamma-\delta W^2}{W^2}\left(\frac{3}{2}+\frac{1}{2}\frac{1}{\gamma}\delta W^2\right)
		+ \frac{3}{2}\frac{\gamma-\delta W^2}{W^2}\\
		&\,\, = -\frac{\delta}{2\gamma} (\gamma-\delta W^2)\le 0.
		\end{split}
	\end{equation}
	
	\begin{lemma} We have in $\mathcal{U}_{R,T ' }$ that 
		\begin{equation}\label{eq_5.27}
			\left(\partial_t-\Delta\right)f\leq -\frac{1}{\psi}\langle\nabla f,\nabla\psi\rangle-af^2+bf+2(C+n\sigma |\bar R (\cdot, N, N, \cdot)|)\sqrt\psi \sqrt f 
		\end{equation}
		where $a=2\delta -\sigma \varepsilon(1-\delta)>0$ and $b = 2(\widetilde C +\widetilde \delta L)+\frac{\sigma}{\varepsilon}$. Here, $C$ and $\widetilde C$ are non-negative constants depending on $\varrho$ and its derivatives.
	\end{lemma}
	\begin{proof}
		Mimicking Lemma 8 in \cite{Borisenko:12}, we write the evolution of $f$ is terms of \eqref{eq_3.27} and \eqref{eq_5.1} as follows 
		\begin{equation}\label{eq_5.28}
			\begin{split}
				\left(\partial_t-\Delta\right)f
				&=|A|^2\psi'(W)\left(\partial_t-\Delta\right)W-|A|^2\psi''(W)|\nabla W|^2
                \\
                &+\psi(W)\left(\partial_t-\Delta\right)|A|^2\,\,-2\langle\nabla\psi,\nabla|A|^2\rangle.
			\end{split}
		\end{equation}
		However
		\begin{equation}\label{eq_5.29}
			\begin{split}
				&2\langle \nabla \psi, \nabla |A|^2\rangle = 	\frac{1}{\psi}\langle \nabla\psi, \nabla f\rangle -\frac{1}{\psi} |\nabla \psi|^2 |A|^2 + \langle \nabla \psi, \nabla |A|^2\rangle\\
				&\,\, \ge \frac{1}{\psi}\langle \nabla\psi, \nabla f\rangle -\frac{1}{\psi} |\nabla \psi|^2 |A|^2  -   \frac{1}{2\psi}|A|^2 |\nabla \psi|^2 -   2\psi |\nabla |A||^2.
			\end{split}
		\end{equation}
		Using  Kato's inequality $|\nabla|A||^2\leq |\nabla A|^2$ we conclude that
		\begin{equation}\label{eq_5.30}
			-2\langle\nabla\psi,\nabla|A|^2\rangle\leq-\frac{1}{\psi}\langle\nabla f,\nabla\psi\rangle+2\psi|\nabla A|^2+\frac{3}{2}\frac{1}{\psi}|A|^2|\nabla\psi|^2.
		\end{equation}
		Hence, expressions  \eqref{eq_3.27} and \eqref{eq_5.2} yield
		\begin{equation}\label{eq_5.31}
			\begin{split}
				& \left(\partial_t-\Delta\right)f\leq -|A|^2\psi'(W)\left(W(|A|^2+\overline{\Ric}(N,N))+2W^{-1}|\nabla W|^2\right)\\
				& \,\, +2\psi (W) \big(-|\nabla A|^2  -n\sigma ( a_i^sa_{sj} +\bar{R}_{i00j} ) a^{ij} +|A|^4  + |A|^2 \overline{{\rm Ric}}(N,N)
				\\
				& \,\,+g^{k\ell}(\bar\nabla_{i} \bar R_{kj0\ell} + \bar\nabla_k \bar R_{\ell i0j})a^{ij}  + 2  g^{k\ell} (a_{is}  \bar R^s_{kj\ell} +  a_{sk}  \bar R^s_{\ell ij}) a^{ij}\big)\\
				& \,\, 
				-|A|^2\psi''(W)|\nabla W|^2-\frac{1}{\psi}\langle\nabla f,\nabla\psi\rangle+2\psi(W)|\nabla A|^2+\frac{3}{2}\frac{1}{\psi}|A|^2|\nabla\psi|^2
			\end{split}
		\end{equation}
		where $'$ denotes derivatives with respect to $W$. Grouping similar terms, one obtains
		\begin{equation}\label{eq_5.32}
			\begin{split}
				& \left(\partial_t-\Delta\right)f\leq  (|A|^4 + |A|^2   \overline{{\rm Ric}}(N,N)) (2\psi(W)-\psi'(W)W) \\
				& \,\, + 2\psi(W)\big(-n\sigma ( a_i^sa_{sj} +\bar{R}_{i00j} ) a^{ij}  
				+g^{k\ell}(\bar\nabla_{i} \bar R_{kj0\ell} + \bar\nabla_k \bar R_{\ell i0j})a^{ij} \\
				& \,\, + 2  g^{k\ell} (a_{is}  \bar R^s_{kj\ell} +  a_{sk}  \bar R^s_{\ell ij}) a^{ij}\big)-\frac{1}{\psi}\langle\nabla f,\nabla\psi\rangle\\
				& \,\, - 2 |A|^2\frac{\psi'(W)}{W}|\nabla W|^2 - |A|^2\psi''(W)|\nabla W|^2 +\frac{3}{2}\frac{\psi'^2(W)}{\psi(W)}|A|^2|\nabla W|^2.
			\end{split}
		\end{equation}
		A straightforward but lengthy computation allows us to verify that 
		\begin{equation}\label{eq_5.33}
			\begin{split}
				g^{k\ell}(\bar\nabla_{i} \bar R_{kj0\ell} + \bar\nabla_k \bar R_{\ell i0j})a^{ij}  + 2  g^{k\ell} (a_{is}  \bar R^s_{kj\ell} +  a_{sk}  \bar R^s_{\ell ij}) a^{ij}
				\leq C |A|+\widetilde C |A|^2
			\end{split}
		\end{equation}
		where the constants $C$ and $\widetilde C $ depends on $\varrho$ and its derivatives and on the curvature of $M$. More explicitly, we have
		\begin{equation}\label{eq_5.34}
			C = C\left(\frac{|\bar\nabla \varrho|}{\varrho}, \frac{|\bar\nabla^2\varrho|}{\varrho},  \frac{|\bar\nabla^3 \varrho|}{\varrho}, |\nabla^M R^M|\right)
		\end{equation}
		and
		\begin{equation}\label{eq_5.35}
			\widetilde C = \widetilde C\left(\frac{|\bar\nabla \varrho|}{\varrho}, \frac{|\bar\nabla^2\varrho|}{\varrho},  \frac{|\bar\nabla^3 \varrho|}{\varrho}, |R^M|\right),
		\end{equation}
		where $R^M$ is the Riemann curvature tensor in $(M, g)$.  Observing that
		\begin{equation}\label{eq_5.36}
			\begin{split}
			&-2|A|^2\frac{\psi'(W)}{W}|\nabla W|^2-|A|^2\psi''(W)|\nabla W|^2+\frac{3}{2}\frac{\psi'^2(W)}{\psi}|A|^2|\nabla W|^2\\
			& \,\, =-\left(2\frac{\psi'(W)}{W}+\psi''(W)-\frac{3}{2}\frac{\psi'^2(W)}{\psi(W)}\right)|A|^2|\nabla W|^2\leq 0
			\end{split}
		\end{equation}
		one concludes that
		\begin{equation}\label{eq_5.37}
			\begin{split}
				& \left(\partial_t-\Delta\right)f\leq (|A|^4 + |A|^2   \overline{{\rm Ric}}(N,N)) (2\psi(W)-\psi'(W)W)\\
				& \,\,+2\psi (W)(C|A|+\widetilde C |A|^2)-2n\sigma\psi(W)\left(a^{ij} a_{i}^sa_{sj}+a^{ij}\bar{R}_{i00j}\right)-\frac{1}{\psi}\langle\nabla f,\nabla\psi\rangle. 
			\end{split}
		\end{equation}
		Given a constant $\varepsilon>0$ to be chosen later one has
		\begin{equation}\label{eq_5.38}
			2 a^{ij} a_i^s a_{sj} \le 2|A|^3 \le \varepsilon |A|^4 + \frac{1}{\varepsilon} |A|^2.
		\end{equation}
		Therefore
		\begin{align}
		    \label{eq_5.39}
				& \left(\partial_t-\Delta\right)f\leq |A|^4 ((2+\sigma\varepsilon)\psi(W)-\psi'(W)W) + |A|^2   \overline{{\rm Ric}}(N,N) (2\psi(W) \nonumber \\
				& \,\,-\psi'(W)W)+2\psi (W)(C|A|+\widetilde C|A|^2) +\frac{1}{\varepsilon}\sigma \psi(W) |A|^2 
				-2n\sigma \psi(W) a^{ij}\bar{R}_{i00j} 
                \\
                &-\frac{1}{\psi}\langle\nabla f,\nabla\psi\rangle. 
		\end{align}
		Note that  \eqref{eq_5.18} implies that
		\begin{equation}\label{eq_5.40}
			\begin{split}
				& \left((2+\sigma \varepsilon)\psi(W)-W\psi'(W)\right)|A|^4=\left(\frac{2+\sigma \varepsilon}{\psi(W)}-W\frac{\psi'(W)}{\psi^2(W)}\right)\psi^2|A|^4\\
				& \,\, \left(\sigma\varepsilon \left(\frac{\gamma}{W^2}-\delta\right)- 2\delta\right) f^2 \leq(\sigma \varepsilon(\gamma\varrho^2-\delta)-2\delta)f^2 
			\end{split}
		\end{equation}
		as well as
		\begin{equation}\label{eq_5.41}
			\begin{split}
				& \left(2\psi(W)-W\psi'(W)\right)|A|^2\overline{\Ric}(N,N)=\text{\small $\left(\frac{2}{\psi(W)}-W\frac{\psi'(W)}{\psi^2(W)}\right)\overline{\Ric}(N,N)\,\psi^2|A|^2$}\\
				& \,\, = -2\delta\overline{\Ric}(N,N) \psi f.
			\end{split}
		\end{equation}
		It follows that
		\begin{equation}\label{eq_5.42}
			\begin{split}
				& \left(\partial_t-\Delta\right)f \leq(\sigma \varepsilon(\gamma\varrho^2-\delta)-2\delta)f^2 -2\delta\overline{\Ric}(N,N) \psi f+2C\psi|A|+2\widetilde Cf
				\\
				& \,\,+\frac{\sigma}{\varepsilon}f+2n\sigma |\bar{R}(\cdot, N, N, \cdot)|  |A| \psi  -\frac{1}{\psi}\langle\nabla f,\nabla\psi\rangle. 
			\end{split}
		\end{equation}
		Since $\overline{\Ric}\geq -L$ for some $L\geq 0$  and $\delta \psi\ge  \widetilde\delta$ we obtain
		\begin{equation}\label{eq_5.43}
			2\widetilde C-2\delta\psi\,\overline{\Ric}(N,N)+\frac{\sigma}{\varepsilon}\leq 2(\widetilde C +\widetilde \delta L)+\frac{\sigma}{\varepsilon}=b.
		\end{equation}
		Since $\gamma\varrho^2 \le 1$ denoting $a=2\delta -\sigma \varepsilon(1-\delta)$ one has
		\begin{equation}\label{eq_5.44}
			\left(\partial_t-\Delta\right)f\leq -\frac{1}{\psi}\langle\nabla f,\nabla\psi\rangle-af^2+bf+2(C+n\sigma |\bar R (\cdot, N, N, \cdot)|)\sqrt\psi \sqrt f 
		\end{equation}
		what ends the proof.
	\end{proof}
	
	\vspace{3mm}
	
	\begin{proposition}\label{curv-est} Let $R'\in (0, R)$ be fixed so that \eqref{eq_5.21} holds. Then
		the norm of the Weingarten map $A$ and its covariant derivatives are bounded in $B_{R'} (o) \times [0,T]$ by
		constants that depend on ${\sup}_{B_R(o)} W^2(\cdot, 0)$ and on the  geometric data ${\sup}_{B_R(o)}\varrho$, $\xi(R)$, $\zeta(R)$, $C$, $\widetilde C$ and $L$.  
	\end{proposition}
	
	\begin{proof} Now, we consider the function
		\begin{equation}\label{eq_5.45}
			\phi(\widetilde\Psi(x,t))=(\zeta(R)-\zeta(\widetilde\Psi(x,t))-C_R t)^2
		\end{equation}
		defined in the set
		\begin{equation}\label{eq_5.46}
			\mathcal{U}_{R, T '} = {\bigcup}_{t\in [0, T '] }  \big\{ y = \widetilde\Psi(x,t) : \zeta(\widetilde\Psi(x,t)) + C_R t \le \zeta(R)\big\}.
		\end{equation}
		Using \eqref{eq_3.13} and \eqref{eq_3.109}  one obtains
		\begin{equation}\label{eq_5.47}
			\begin{split}
				& \left(\partial_t-\Delta\right)\phi\leq -2(\zeta(R)-\zeta -C_R t) (\partial_t\zeta -\Delta\zeta +C_R) - 2 |\nabla \zeta|^2\\
				& \,\, \le 2(\zeta(R)-\zeta -C_R t) (n \xi'(r) + n\sigma\xi(r) \langle\bar\nabla r, N\rangle -C_R) - 2 |\nabla \zeta|^2.
			\end{split}
		\end{equation}
		Since 
		\begin{equation}\label{eq_5.48}
			C_R =  {\sup}_{B_R(o)} (\xi'+ \sigma \xi)
		\end{equation}
		we have 
		\begin{equation}\label{eq_5.49}
		\left(\partial_t-\Delta\right)\phi\leq -
		2|\nabla\zeta|^2. 
		\end{equation}
		Therefore we compute
		\begin{equation}
			\begin{split}\label{eq_5.50}
			& \left(\partial_t-\Delta\right)(\phi f)=f\left(\partial_t-\Delta\right)\phi +\phi\left(\partial_t-\Delta\right)f-2\langle\nabla\phi,\nabla f\rangle\\
			& \,\, \le -2f |\nabla\zeta|^2-
			\Big\langle \nabla(\phi f)-f\nabla\phi,\frac{\nabla\psi}{\psi}\Big\rangle-a\phi f^2
			+b\phi f\\ 
			&\,\,+2\phi(C+|\bar{R}(\cdot, N, N, \cdot)|)\sqrt{\psi }\sqrt f-2\Big\langle \frac{\nabla\phi}{\phi},\nabla(\phi f)-f\nabla\phi\Big\rangle.
			\end{split}
		\end{equation}
		However we have
		\begin{equation}\label{eq_5.51}
			-2|\nabla\zeta|^2 +2\frac{|\nabla\phi|^2}{\phi}= - 2 |\nabla \zeta|^2 + 8 \frac{(\zeta(R)-\zeta(r) -C_R t)^2}{(\zeta(R)-\zeta(r) -C_R t)^2}|\nabla \zeta|^2 =6|\nabla\zeta|^2.
		\end{equation}
		Hence,
		\begin{equation}\label{eq_5.52}
			\begin{split}
			& \left(\partial_t-\Delta\right)(\phi f)\leq 
			6|\nabla\zeta|^2f-\Big\langle \nabla(\phi f),\frac{\nabla\psi}{\psi}+2\frac{\nabla\phi}{\phi}\Big\rangle+\Big\langle f\nabla\phi,\frac{\nabla\psi}{\psi}\Big\rangle\\
			&\,\,\,\,-a\phi f^2+b\phi f+2\phi(C+\sigma |\bar R (\cdot, N, N, \cdot)|)\sqrt{\psi}\sqrt f.
			\end{split}
		\end{equation}
		Since that $\nabla\psi=\psi'(W)\nabla W$ and
		\begin{equation}\label{eq_5.53}
			\nabla W^{-1}=\nabla\langle X,N\rangle=\langle X,N\rangle (\bar\nabla\log\varrho)^\top-\langle\bar\nabla\log\varrho,N\rangle X^\top-AX^\top
		\end{equation}
		we have
		\begin{equation}\label{eq_5.54}
			\frac{\nabla\psi}{\psi} = -\frac{2\gamma}{\gamma-\delta W^2} W\left(\langle X,N\rangle\bar\nabla\log\varrho -\langle\bar\nabla\log\varrho ,N\rangle X-AX^\top\right).
		\end{equation}
		Hence,
		\begin{equation}\label{eq_5.55}
			\begin{split}
				\left|\frac{\nabla\psi}{\psi}\right| \le 
				4\frac{\gamma}{\delta} |\bar\nabla \varrho|+  2\gamma \varrho \frac{\sqrt{\psi(W)}}{W} \sqrt  f \leq 4\delta^{-1}{\sup}_{\mathcal{C}_R}|\bar{\nabla}\log\varrho|+4\sqrt{\psi}\sqrt f.
			\end{split}
		\end{equation}
		where we used that $\psi(W) \le \delta^{-1}$. Denoting  $c=4\delta^{-1}{\sup}_{B_R(o)} |\bar{\nabla}\log\varrho|$ 
		we have
		\begin{equation}\label{eq_5.56}
			\begin{split}
			& \left(\partial_t-\Delta\right)(\phi f)\leq 
			6|\nabla\zeta|^2f-\Big\langle \nabla(\phi f),\frac{\nabla\psi}{\psi}+2\frac{\nabla\phi}{\phi}\Big\rangle+|\nabla\phi|\big(c+4\sqrt{\psi}\sqrt f\big)f\\
			& -a\phi f^2+b\phi f
			+2\phi(C+|\bar{R}(\cdot, N, N, \cdot)|)\sqrt{\psi}\sqrt f.
			\end{split}
		\end{equation}
		We conclude that at a point where  $\phi f$ attains a maximum value in  $\mathcal{U}_{R', T_R }\subset \mathcal{U}_{R,T ' }$ it holds (in case $t\neq 0$) that 
		\begin{equation}\label{eq_5.57}
			\begin{split}
				& a\phi f^2\leq 4|\nabla\phi| \sqrt{\psi}\sqrt f f
				+ \left(6|\nabla\zeta|^2+ b\phi +c|\nabla\phi|\right)f\\
                & \,\, +2(C+|\bar {R}(\cdot, N, N, \cdot)|)\sqrt{\phi\psi}\sqrt{\phi f}
			\end{split}
		\end{equation}
		so multiplying by $\sqrt{\phi}/\sqrt{f}$ and grouping the terms
		\begin{equation}\label{eq_5.58}
			\begin{split}
			& a(\sqrt{\phi f})^3\leq 4|\nabla\phi|\sqrt{\psi}\sqrt{\phi}f
			+\left(6|\nabla\zeta|^2+
			b\phi +c|\nabla\phi|\right)\sqrt{\phi f}\\
			&\,\,\,\,
			+2(C+|\bar {R}(\cdot, N, N, \cdot)|)\phi\sqrt{\psi}\sqrt{\phi}.
			\end{split}
		\end{equation}
		Considering that 
		\begin{equation}\label{eq_5.59}
			\nabla\phi=-2\sqrt{\phi}\,\nabla \zeta = -2\sqrt{\phi}\,\xi(r)\nabla r\;\;\;\;\mbox{and}\;\;\;\;\phi\leq \zeta^2(R)\;\;\;\;\mbox{and}\;\;\;\;\sqrt{\psi}\leq\frac{1}{\sqrt{\delta}}
		\end{equation}
		and using that $|\nabla r|\le 1$ one concludes that
		\begin{equation}\label{eq_5.60}
			\begin{split}
			& a(\sqrt{\phi f})^3\leq\frac{8}{\sqrt{\delta}} \xi(r) (\sqrt{\phi f})^2 +\left(6\xi^2(r)+2c\xi(r) \zeta(R)+b\zeta^2(R)\right)(\sqrt{\phi f})\\
			&\,\,+\frac{2}{\sqrt{\delta}}(C+|\bar {R}(\cdot, N, N, \cdot)|)\zeta^3(R).
			\end{split}
		\end{equation}
		Therefore
		\begin{equation}\label{eq_5.61}
			\text{\small $a\left(\frac{\sqrt{\phi f}}{\zeta(R)}\right)^3-\frac{8}{\sqrt{\delta}}\frac{\xi(R)}{\zeta(R)}\left(\frac{\sqrt{\phi f}}{\zeta(R)}\right)^2 -\left(6\frac{\xi^2(R)}{\zeta^2(R)}+2c\frac{\xi(R)}{\zeta(R)}+b\right)\left(\frac{\sqrt{\phi f}}{\zeta(R)}\right)
			-\frac{2}{\sqrt{\delta}}C\leq 0.$}
		\end{equation}
		In this case, either
		\begin{equation}\label{eq_5.62}
			a\left(\frac{\sqrt{\phi f}}{\zeta(R)}\right)^3-\frac{2}{\sqrt{\delta}}C \le \delta
			\left(\frac{\sqrt{\phi f}}{\zeta(R)}\right)^3 
		\end{equation}
		or 
		\begin{equation}\label{eq_5.63}
			\delta
			\left(\frac{\sqrt{\phi f}}{\zeta(R)}\right)^3  -\frac{8}{\sqrt{\delta}}\frac{\xi(R)}{\zeta(R)}\left(\frac{\sqrt{\phi f}}{\zeta(R)}\right)^2-\left(6\frac{\xi^2(R)}{\zeta^2(R)}+2c\frac{\xi(R)}{\zeta(R)}+b\right)\left(\frac{\sqrt{\phi f}}{\zeta(R)}\right)  \leq 0.
		\end{equation}
		We conclude that
		\begin{equation}\label{eq_5.64}
			\begin{split}
				&  \frac{\sqrt{\phi f}}{\zeta(R)} \leq C_1:=\max\Bigg\{ \frac{(2C)^{\frac{3}{2}}}{\sqrt\delta},  \frac{4}{\delta^{\frac{3}{2}}}\frac{\xi(R)}{\zeta(R)} + \left( \frac{16}{\delta^3}\frac{\xi^2(R)}{\zeta^2(R)} + \frac{1}{\delta}\left(6\frac{\xi^2(R)}{\zeta^2(R)} + 2c\frac{\xi(R)}{\zeta(R)}+b\right)\right)^{\frac{1}{2}}.
			\end{split}
		\end{equation}
		Therefore 
		\begin{equation}\label{eq_5.65}
			{\inf}_{B_{R'}(o) \times [0, T]}\left(1-\frac{\zeta(r)}{\zeta(R)}-\frac{C_R}{\zeta(R)} t\right)\sqrt{f}\leq  C_1.
		\end{equation}
		Since $f=\psi (W)|A|^2$ the choice of $R'\in (0,R)$ in \eqref{eq_5.21} implies that
		\begin{equation}\label{eq_5.66}
			{\inf}_{B_{R'}(o) \times [0, T]}\left(1-\frac{\zeta(r)}{\zeta(R)}-\frac{C_R}{\zeta(R)} t\right) \ge \frac{1}{2}
		\end{equation}
		we have
		\begin{equation}\label{eq_5.67}
			{\sup}_{B_{R'}(o)\times [0,T]}|A|\leq  2   \sqrt{1-\delta} \,  C_1\le 2C_1.
		\end{equation}
		This shows that $|A|$ is bounded in $B_{R'}(o)\times [0, T]$ by some constant that depends on ${\sup}_{B_R(o)} W^2(\cdot, 0)$ and on the  geometric data ${\sup}_{B_R(o)}\varrho$, $\xi(R)$, $\zeta(R)$, $C$, $\widetilde C$ and $L$.  
		Since
		\begin{equation}\label{eq_5.68}
			\widetilde\delta = \frac{\delta}{1-\delta}  =\frac{\gamma}{2\,{\rm sup}_{B_{R'}(o)\times [0,T]}W^2- \gamma}\le \frac{\gamma}{2\,{\rm inf}_{B_{R'}(o)}\varrho^{-2}- \gamma}  = \frac{\gamma}{\gamma}=1,
		\end{equation}
		we have $b=2(\widetilde C + \widetilde \delta L) \le  2(\widetilde C +  L).$ We also recall that 
		\begin{equation}\label{eq_5.70}
			c = \frac{4}{\delta}\, {\sup}_{B_R(o)} \frac{|\bar\nabla\varrho|}{\varrho}\cdot
		\end{equation}
		Hence we conclude that there exist positive constants $C, C'$ depending on the geometric data listed above such that
		\begin{equation}\label{eq_5.71}
			|A|(o,t)  \le \big(C \,{\rm sup}_{B_R(o)\times [0, T]} W + C' \, {\rm sup}_{B_R(o)\times [0, T]} W^{3}\big) \frac{\xi(R)}{\zeta(R)}
		\end{equation}
		for $t\in [0, T]$. This finishes the proof of the desired curvature estimate.

\medskip
        
		\par From this estimate we can conclude that  the covariant derivatives of $A$ are also bounded. Indeed, proceeding inductively as in  \cite{Ecker:91} and \cite{Borisenko:12} one supposes that
		for each $k =0, 1,\ldots, \ell-1$   there exists a constant $C_k = C_k (R, T)$ so that
		\begin{equation}\label{eq_5.72}
			|\nabla^kA|\leq C_k  \;\;\;\mbox{for}\;\;\;k=0,1,\cdots,\ell-1
		\end{equation}
		where $C_k$ depends on the bounds of $|\nabla^m A|$  and on the tensors $\bar\nabla^m \bar R$ for $0\leq m\leq k-1$ in $B_R(o)\times [0, T]$ and on the same sort of geometric data as above. 
		\par As in  \cite{Ecker:91} and \cite{Borisenko:12}
		we are going to use variants of the Simons' inequality for higher order covariant derivatives of $A$ which have the form
		\begin{equation}\label{eq_5.73}
			\frac{1}{2}(\partial_t- \Delta)|\nabla^\ell A|^2+|\nabla^{\ell+1}A|^2\le D_\ell (|\nabla^\ell A|^2+1)
		\end{equation}
		where the constant $D_\ell$ depends on the bounds of $|\nabla^kA|$  and on the tensors $\bar\nabla^k \bar R$ for $0\leq k\leq \ell-1$ in $B_R(o)\times [0, T]$.
		Now, we define
		\begin{equation}\label{eq_5.74}
		h=|\nabla^\ell A|^2+\lambda|\nabla^{\ell -1}A|^2
		\end{equation}
		where $\lambda$ is a constant to be specified later. Setting $\lambda\ge 2D_\ell$ one obtains
		\begin{equation}\label{eq_5.75}
			\begin{split}
				\frac{1}{2}\partial_t  h  \le  \frac{1}{2}\Delta h   - \frac{\lambda}{2}  h + \frac{\lambda^2}{2} |\nabla^{\ell-1} A|^2 +  D_{\ell-1} |\nabla^{\ell-1} A|^2 + D_\ell + \lambda D_{\ell-1}.
			\end{split}
		\end{equation}
		Choosing $\lambda^2 \ge 2D_{\ell-1}$ we conclude that 
		\begin{equation}\label{eq_5.76}
			(\partial_t - \Delta) h \le  -\lambda  h + \lambda^2 \widetilde{C}_\ell + \widetilde D_{\ell},
		\end{equation}
		where $\widetilde D_\ell = 2D_\ell + 2\lambda D_{\ell-1}$ and  $\widetilde C_\ell =  2|\nabla^{\ell-1} A|^2$.
		Proceeding similarly as above one computes
		\begin{equation}\label{eq_5.77}
			\begin{split}
				\left(\partial_t-\Delta\right)(\phi h)\leq- 2 h |\nabla\zeta|^2 + (-\lambda  h + \lambda^2 \widetilde{C}_\ell + \widetilde D_{\ell})\phi-2\left\langle\phi^{-1}\nabla\phi,\nabla(\phi h)- h\nabla\phi\right\rangle.
			\end{split}
		\end{equation}
		Therefore
		\begin{equation}\label{eq_5.78}
			\begin{split}
				& \left(\partial_t-\Delta\right)(\phi h)+2\left\langle\phi^{-1}\nabla\phi,\nabla(\phi h)\right\rangle \le - 2 h |\nabla\zeta|^2 + 2\phi^{-1}|\nabla \phi|^2 h \\
                & \,\,+ (-\lambda  h + \lambda^2 \widetilde{C}_\ell + \widetilde D_{\ell})\phi.
			\end{split}
		\end{equation}
		Using again that 
		\[
		-2|\nabla \zeta|^2+ 2\phi^{-1}|\nabla \phi|^2 = 6 |\nabla \zeta|^2
		\]
		one concludes that
		\begin{equation}\label{eq_5.79}
			\begin{split}
				& \left(\partial_t-\Delta\right)(\phi h)+2\left\langle\phi^{-1}\nabla\phi,\nabla(\phi h)\right\rangle \le  6 |\nabla\zeta|^2 h - \lambda \phi h + (\lambda^2 \widetilde{C}_\ell + \widetilde D_{\ell})\phi.
			\end{split}
		\end{equation}
		We have at a maximum point of  $\phi h$ that
		\begin{equation}\label{eq_5.80}
			(\lambda\phi-6\xi^2(r))h\leq (\lambda^2 \widetilde{C}_\ell + \widetilde D_{\ell})\zeta^2(R).
		\end{equation}
		Since that $\phi\geq\zeta^2(R)/4$ in  $B_{R'} (o) \times [0, T]$ we have
		\begin{equation}\label{eq_5.81}
			\left( \frac{\lambda}{4}-6\frac{\xi^2(R)}{\zeta^2(R)}\right) h \leq \lambda^2 \widetilde{C}_\ell + \widetilde D_{\ell}.
		\end{equation}
		Setting
		\begin{equation}\label{eq_5.82}
			\lambda\ge \max\left\{ 2 D_\ell, (2D_{\ell-1})^{1/2}, 48  \frac{\xi^2(R)}{\zeta^2(R)}\right\}
		\end{equation}
		one obtains
		\begin{equation}\label{eq_5.83}
		h\leq \frac{1}{6} \frac{\zeta^2(R)}{\xi^2(R)} \big(\lambda^2 \widetilde{C}_\ell + \widetilde D_{\ell}\big).
		\end{equation}
		We conclude that
		\begin{equation}\label{eq_5.84}
			|\nabla^\ell A|^2\leq \frac{1}{6} \frac{\zeta^2(R)}{\xi^2(R)} \big(2\lambda^2 |\nabla^{\ell-1}A|^2
			+ \widetilde D_{\ell}\big) - \lambda |\nabla^{\ell-1}A|^2.
		\end{equation}
		A suitable choice of a large enough $\lambda$ yields the desired estimate. This finishes the proof.
	\end{proof}

	\section{Proof of the main existence result}\label{proof-thm}
	
	In this section we assemble the results obtained previously to guarantee the existence of solution to the modified mean curvature flow.
	
	\begin{proof}[Proof of Theorem \ref{thm_0.1}] We consider an exhaustion $\{B_{R_k}(o)\}_{k\in \mathbb{N}}$ of $M$ by geodesic balls  centered at a pole $o\in M$ whose radius form a increasing sequence with $R_k \to \infty$ as $k\to\infty$.  The  gradient \emph{a priori} estimates in propositions \ref{evolutionW}, \ref{intest} and \ref{prop_4.1}  imply that the equation is uniformly parabolic in $B_{R_k'} (o) \times [0, T]$ for any fixed $T>0$. Hence, Theorem \ref{local-thm} guarantees the existence 
        of a unique solution \( u_k \), defined in \( B_{R_k'}(o) \times [0, \infty) \) of the following equation
		\begin{equation}\label{eq_6.2}
			\partial_t u = W\bigg(\operatorname{div}\left(\frac{\nabla u}{W}\right)+\left\langle\nabla\log\varrho,\frac{\nabla u}{W}\right\rangle-n\sigma\bigg)
		\end{equation}
		such that
		\begin{equation}\label{eq_6.3}
			\begin{cases} 
				u(\cdot,0)= u_0(\cdot) & \quad \mbox{in}\quad B_{R_k'}(o)\times\{0\}\\
				u(x,t)=u_0(x) & \quad  \mbox{on}\quad \partial B_{R_k'}(o)\times[0,\infty).
			\end{cases}
		\end{equation}
		\par Now we have to obtain uniform \emph{a priori} estimates for the higher derivatives of the local solutions $u_k$. In order to do that, fix $T>0$.  Given $k\in \mathbb{N}$, let $R_k' \in (0, R_k)$ such that 
		\begin{equation}\label{eq_6.1}
			\zeta(r) + C_{R_k} T\le \frac{1}{2}\zeta(R_k)
		\end{equation}
		for $r< R_k'$. The existence of such a sequence is granted for any $T$ since $\zeta$ and $R_k$ are increasing. Propositions \ref{evolutionW}, \ref{intest} and \ref{prop_4.1}  guarantee that any solution of $\eqref{eq_6.2}$ on a time interval $[0,T]$ must satisfy
		\begin{equation}\label{eq_6.4}
			\sup_{{\overline B}_{R_k'}(o)\times[0,T]}|\nabla u_{k}|\leq C_0
		\end{equation}
		for some constant $C_0$.
	Hence, for the pair $(x,u_k(x,t))\in \mathcal{U}_{R_k',T}$ we use the estimates of $W,|A|$ and $|\nabla^\ell A|$ (see proposition \ref{curv-est}) to obtain for any integer $l\geq 0$ a constant $C_\ell$ such that  
    \begin{equation}
	\sup_{{\overline{B}_{R_k'}(o)}\times[0,T]}|\nabla^\ell u_k|\leq C_\ell
    \end{equation}
    for some constant $C_\ell$ depending on $C_0$. Thence, we guarantee uniform \emph{a priori} estimates for all derivatives of $u_k$ for each  $B_{R_k'} \times [0, T]$. In particular, for any $k$, the sequence $\{u_\ell\}_{\ell \ge k}$ has a subsequence that converges uniformly in compact subsets of $B_{R_k'}$ to a \emph{smooth} solution $u$ of \eqref{eq_6.2} in $B_{R_k'}\times [0, \infty)$ with $u(\cdot, 0) = u_0(\cdot)$. A usual diagonal process yields
 a solution $u$ of \eqref{eq_6.2} in $M\times [0,\infty)$. 

The general case of Lipschitz initial data can be dealt with by approximation arguments as  in \cite{Ecker:91}, \cite{Unterberger:03} and \cite{Borisenko:12}.
\end{proof}

\section*{Acknowledgements}
J.H.S. de Lira is partially suppoted by CNPq grant PQ-306626/2022-5. M.N. Soares is partially supported by FACEPE grant BFP-0015-1.01/25 and CNPq grant SWP-317751/2023-9. 

\end{document}